\documentclass[12pt]{article}
\usepackage{algorithm}
\usepackage{algpseudocode}
\usepackage{amsfonts}
\usepackage{amsmath}
\usepackage{amssymb} 
\usepackage{amsthm} 
\usepackage[title]{appendix} 
\usepackage{autobreak}
\usepackage{bigdelim} 
\usepackage{blkarray}
\makeatletter
\def\BA@fnsymbol#1{\ensuremath{%
  \ifcase#1\or *\or \dagger\or \ddagger\or \mathsection\or \mathparagraph\or \|\or **\or \dagger\dagger\or \ddagger\ddagger \else@ctrerr\fi}}%
\makeatother
\usepackage{bm}
\usepackage{booktabs}
\usepackage{cancel} 
\usepackage{cases}
\usepackage{chngcntr}
\usepackage{colortbl} 
\usepackage{comment}
\usepackage{empheq} 
\usepackage{enumerate}
\usepackage{fancyhdr} 
\usepackage{fancyvrb} 
\usepackage{framed} 
\usepackage{graphicx}
\usepackage{hyperref} 
\hypersetup{
  colorlinks,
  linkcolor={red!50!black},
  citecolor={blue!50!black},
  urlcolor={blue!80!black}
}
\usepackage{mathrsfs} 
\usepackage{multirow} 
\usepackage[sort&compress,numbers]{natbib}
\usepackage{setspace} 
\usepackage{siunitx} 
\usepackage{subfig} 
\usepackage{pdflscape} 
\usepackage{pgfplots} 
\usepackage{pxrubrica} 
\usepackage{url}
\usepackage[colorinlistoftodos,prependcaption,textsize=footnotesize]{todonotes}

\usepackage[top=35truemm,bottom=35truemm,left=30truemm,right=30truemm]{geometry} 
\makeatletter

\@addtoreset{equation}{section}
\makeatother

\algtext*{EndWhile}
\algtext*{EndIf}
\algtext*{EndFor}

\theoremstyle{plain}
\newtheorem{theorem}{Theorem}[section]
\newtheorem{proposition}[theorem]{Proposition}%

\newtheorem{lemma}[theorem]{Lemma}

\theoremstyle{definition}
\newtheorem{assumption}[theorem]{Assumption}
\theoremstyle{remark}
\newtheorem{remark}[theorem]{Remark}%

\DeclareMathOperator*{\dom}{dom}

\DeclareMathOperator*{\dist}{dist}

\DeclareMathOperator*{\vol}{vol}

\DeclareMathOperator{\subjectto}{subject~to}

\newcommand{\argmin}{\mathop{\rm arg~min}\limits}
\newcommand{\minimize}{\mathop{\rm minimize}\limits}

\newcommand{\relmiddle}[1]{\mathrel{}\middle#1\mathrel{}}

\newcommand{\abs}[1]{\lvert #1 \rvert} 
\newcommand{\norm}[1]{\lVert #1 \rVert} 

\newcommand{\bbK}{\mathbb{K}}

\newcommand{\bbR}{\mathbb{R}}

\newcommand{\bbV}{\mathbb{V}}

\newcommand{\calS}{\mathcal{S}}

\newcommand{\COP}{\mathcal{COP}}

\newcommand{\RNum}[1]{\uppercase\expandafter{\romannumeral #1\relax}} 
\newcommand{\Rnum}[1]{\lowercase\expandafter{\romannumeral #1\relax}} 

\title{Inexact subgradient algorithm with a non-asymptotic convergence guarantee for copositive programming problems}

\makeatletter
\let\@fnsymbol\@arabic
\makeatother

\author{
\normalsize
    Mitsuhiro Nishijima\thanks{Department of Industrial and Systems Engineering, Keio University, 3-14-1 Hiyoshi, Kohoku-ku, Yokohama-shi, 2238522, Kanagawa, Japan. ({\tt nishijima@keio.jp}).
}
\and
\normalsize
    Pierre-Louis Poirion\thanks{Center for Advanced Intelligence Project, RIKEN, 1-4-1, Nihonbashi, Chuo-ku, 1030027, Tokyo, Japan. ({\tt pierre-louis.poirion@riken.jp}).
}
\and
\normalsize
        Akiko Takeda\thanks{Graduate School of Information Science and Technology, The University of Tokyo, 7-3-1 Hongo, Bunkyo-ku, 1138656, Tokyo, Japan; Center for Advanced Intelligence Project, RIKEN, 1-4-1, Nihonbashi, Chuo-ku, 1030027, Tokyo, Japan ({\tt takeda@mist.i.u-tokyo.ac.jp}).
}
}

\begin{document}
\maketitle

\begin{abstract}\noindent
In this paper, we propose a subgradient algorithm with a non-asymptotic convergence guarantee to solve copositive programming problems.
The subproblem to be solved at each iteration is a standard quadratic programming problem, which is NP-hard in general.
However, the proposed algorithm allows this subproblem to be solved inexactly.
For a prescribed accuracy $\epsilon > 0$ for both the objective function and the constraint arising from the copositivity condition, the proposed algorithm yields an approximate solution after $O(\epsilon^{-2})$ iterations, even when the subproblems are solved inexactly.
We also discuss exact and inexact approaches for solving standard quadratic programming problems and compare their performance through numerical experiments.
In addition, we apply the proposed algorithm to the problem of testing complete positivity of a matrix and derive a sufficient condition for certifying that a matrix is not completely positive.
Experimental results demonstrate that we can detect the lack of complete positivity in various doubly nonnegative matrices that are not completely positive.
\end{abstract}
\vspace{0.5cm}


\section{Introduction}\label{sec:intro}
A copositive programming problem is a conic linear programming problem involving a copositive cone.
Many studies have shown that various computationally challenging optimization problems can be reformulated as copositive programming problems in a unified manner~\cite{BDd+2000,BG2023,Burer2009}.
Motivated by this, researchers have proposed various approaches to solving copositive programming problems.

Approximations for copositive cones are widely used to solve copositive programming problems.
Typically, the cone consisting of sums of a real symmetric positive semidefinite matrix and a symmetric entrywise nonnegative matrix provides an inner approximation to a copositive cone.
However, except in a few specific situations~\cite{GY2022,KKT2020}, it is difficult to assess the accuracy of this approximation theoretically.
To approximate copositive cones to arbitrary accuracy, various approximation hierarchies have been proposed~\cite{AM2019,dP2002,Lasserre2014,LV2022,Parrilo2000_Semidefinite,PVZ2007,SBD2012,Yildirim2012}.
Bundfuss and D\"{u}r~\cite{BD2009} provided an adaptive approximation scheme for copositive cones based on a simplicial partition and showed its asymptotic convergence.
\v{Z}ilinskas~\cite{Zilinskas2011} also utilized a simplicial partition to solve a copositive programming problem with a single scalar variable.
Moreover, cutting-plane methods have been used to solve copositive programming problems~\cite{Bd2022_An,GBT2026}, although they lack theoretical results on asymptotic or non-asymptotic convergence.

Copositive programming is a subclass of convex semi-infinite programming, so we can apply methods for solving convex semi-infinite programming problems to copositive programming problems~\cite{GL2018}.
Ahmed, D\"{u}r, and Still~\cite{ADS2013} interpreted the approximation schemes for copositive programming proposed in \cite{Bd2002,BD2009} as discretization methods for semi-infinite programming problems and derived their asymptotic convergence rates.
More recently, researchers have established theoretical results for copositive programming through the lens of semi-infinite programming, including strong duality, optimality conditions, and representations of the faces of copositive cones~\cite{GRV20XX,KT2020,KT2022_On_equivalent,KT2022_On_strong,KTD2020}.

In this paper, we also reformulate copositive programming problems as convex semi-infinite programming problems, specifically as convex programming problems with a single nonsmooth functional constraint, and propose a subgradient algorithm (Algorithm~\ref{alg:proposed}) to solve them.
The bottleneck of the proposed algorithm is solving an NP-hard subproblem at each iteration, but it allows the subproblem to be solved inexactly.
The proposed algorithm is based on the subgradient algorithm in \cite[Equation~(3.2.24)]{Nesterov2018}.
If we apply the subgradient algorithm in \cite[Equation~(3.2.24)]{Nesterov2018} to the convex semi-infinite programming problem directly, we have to solve the subproblem exactly to find a constraint that attains the maximal violation.
The subproblem is nonconvex in general and solving it exactly is sometimes demanding, especially for large-scale problems.

In particular, when we apply the proposed algorithm to a copositive programming problem, the subproblem at each iteration reduces to a standard quadratic programming problem.
It is well known that this problem is NP-hard~\cite{MK1987}.
As discussed in Section~\ref{sec:COPP}, there are several approaches to solving this subproblem either exactly or inexactly.
To solve the subproblem exactly, the subproblem can be reformulated as a mixed-integer linear programming problem whose size depends only on the size of the copositive programming problem~\cite{GY2021}.
To solve the subproblem inexactly, one may employ the polynomial-time approximation scheme~\cite{Bd2002}, which requires evaluating only a finite number of objective values and does not require external solvers, as well as its randomized variant.
In addition, we compare their performance through numerical experiments.
The results show that exact computation is acceptable when the size of the problem is small, whereas inexact computation becomes a viable alternative as the size increases.
In contrast to the proposed algorithm, existing approaches for copositive programming problems typically require solving potentially large-scale problems with external solvers.
Many of these methods solve linear programming or semidefinite programming problems as approximations.
The higher the desired accuracy, the larger the variable dimension and the number of constraints in the approximate problems, regardless of the size of the original copositive programming problems.

Moreover, we establish a theoretical non-asymptotic convergence result for the proposed algorithm, which also covers the case where the subproblems are solved inexactly.
In Theorem~\ref{thm:convergence_SIP}, we provide an explicit iteration-complexity bound that guarantees that an approximate solution achieves a given accuracy.
Specifically, for a prescribed accuracy $\epsilon > 0$ with respect to both the objective function and the constraint induced by copositivity, the proposed algorithm yields an approximate solution after $O(\epsilon^{-2})$ iterations.
In contrast, most existing methods either lack theoretical guarantees on the accuracy of approximate solutions or only provide asymptotic convergence.

Our algorithm is closely related to the subgradient algorithms proposed by Nesterov~\cite{Nesterov2004}, Beck et al.~\cite{BBG+2010}, and Wei, Haskell, and Zhao~\cite{WHZ2020} for convex semi-infinite programming problems.
In particular, the algorithm presented by Wei, Haskell, and Zhao~\cite{WHZ2020} also permits inexact computation of subproblems.
The four algorithms differ in the rules for selecting the objective or constraint function to improve at each iteration, choosing the step size, and computing the solution to output.
These differences yield a smaller iteration-complexity bound for our algorithm to achieve a solution within a given accuracy than the others; see Table~\ref{tab:compare_num_iter} for details.
Additionally, our algorithm does not rely on a compactness assumption on the constraint set, whereas the other algorithms use a parameter that depends on such an assumption.

Furthermore, we apply the proposed algorithm to testing complete positivity of a matrix.
Determining whether a matrix is completely positive is NP-hard in general settings~\cite{DG2014}, and developing practical methods for this task has been recognized as an important problem~\cite{BDS2015}.
Various approaches have been proposed to detect complete positivity, some of which rely on specific matrix structures~\cite{DD2012,Xu2001}, while others can handle unstructured matrices~\cite{BR2006,JS2009,Nie2014,Bd2022_An}.
Following~\cite{Bd2022_An}, we formulate the problem of testing complete positivity as a copositive programming problem with a ball constraint.
By applying the proposed algorithm to this formulation, we obtain a sufficient condition ensuring that an input matrix is not completely positive (Theorem~\ref{thm:suff_cond_not_CP}).
Our method does not require any structural assumptions on the input matrix.
In addition, numerical experiments demonstrate that we can detect that certain doubly nonnegative matrices are not completely positive.

The organization of this paper is as follows.
In Section~\ref{sec:preliminaries}, we introduce and recall the notation and concepts used in this paper.
In Section~\ref{sec:SIP}, we propose an inexact subgradient algorithm for solving general convex semi-infinite programming problems and show its non-asymptotic convergence.
In Section~\ref{sec:COPP}, we apply the proposed algorithm to general and specific copositive programming problems and discuss how to implement the algorithm.
In particular, we present several approaches to solving the subproblem, a standard quadratic programming problem, that needs to be solved at each iteration.
In Section~\ref{sec:experiment}, we verify the effectiveness of the methods introduced in the previous sections through numerical experiments.
Finally, Section~\ref{sec:conclusion} concludes the paper.

\section{Preliminaries}\label{sec:preliminaries}
For a positive integer $n$, we define $[n] \coloneqq \{1,\dots,n\}$.
For a finite set $C$, we use $\abs{C}$ to denote the number of elements in $C$.
For a real number $a$, we use $\abs{a}$ to denote the absolute value of $a$.

Let $V$ be a finite-dimensional real vector space equipped with a norm $\norm{\cdot}$ induced by an inner product denoted by $\langle \cdot,\cdot \rangle$.
We define
\begin{equation*}
B(x,r) \coloneqq \{y\in V \mid \norm{x-y} \le r\}
\end{equation*}
as the closed ball with center $x \in V$ and radius $r>0$.
Let $S$ be a non-empty closed convex subset of $V$.
For $x\in V$, the distance between the point $x$ and the set $S$ is defined as
\begin{equation*}
\dist(x,S) \coloneqq \min_{y\in S}\norm{x-y}.
\end{equation*}
The point $y^* \in S$ that attains $\dist(x,S) = \norm{x-y^*}$ is uniquely determined.
We call the point $y^*$ the \emph{projection} of $x$ onto $S$ and denote it by $P_S(x)$.
Let $f\colon V \to (-\infty,+\infty]$ be a convex function.
Let $\dom f \coloneqq \{x \in V \mid f(x) < +\infty\}$ denote the effective domain of $f$.
For $x\in \dom f$, the \emph{subdifferential} of $f$ at $x$ is defined by
\begin{equation*}
\partial f(x) \coloneqq \{d \in V \mid f(y) \ge f(x) + \langle d, y-x\rangle \text{ for all $y\in V$}\}.
\end{equation*}
To specify the variable, we may write $\partial_x f(x)$ for $\partial f(x)$.
The function $f$ is \emph{subdifferentiable} on a set $C$ in $V$ if $\partial f(x)$ is non-empty for all $x\in C$.

We use boldface lowercase letters such as $\bm{a}$ to denote vectors.
For a vector $\bm{a}$, we denote its transpose by $\bm{a}^\top$ and its $i$th element by $a_i$.
We use $\bbR^n$ and $\bbR_+^n$ to denote the space of $n$-dimensional real vectors and the set of entrywise nonnegative vectors in $\bbR^n$, respectively.
The space $\bbR^n$ is equipped with the standard inner product and the $2$-norm, given by $\norm{\bm{a}}_2 \coloneqq \sqrt{\bm{a}^\top \bm{a}}$ for $\bm{a}\in \bbR^n$.
The zero vector is written as $\bm{0}$.
We use $\bm{e}_i$ to denote the vector whose $i$th element is $1$ and the other elements are $0$.
The $(n-1)$-dimensional standard simplex in $\bbR^n$ is defined as
\begin{equation*}
\Delta^{n-1} \coloneqq \left\{\bm{\delta} \in \bbR_+^n \relmiddle| \sum_{i=1}^n \delta_i = 1\right\}.
\end{equation*}

We use boldface uppercase letters such as $\bm{A}$ to denote matrices.
We write $A_{ij}$ for the $(i,j)$th element of a matrix $\bm{A}$.
We define $\calS^n$ to be the space of real $n\times n$ symmetric matrices.
The space $\calS^n$ is equipped with the inner product defined as $\langle \bm{A},\bm{B} \rangle \coloneqq \sum_{i,j=1}^n A_{ij}B_{ij}$ for $\bm{A},\bm{B} \in \calS^n$ and the induced Frobenius norm defined as $\norm{\bm{A}}_{\rm F} \coloneqq \sqrt{\langle \bm{A},\bm{A} \rangle}$.
The zero matrix is written as $\bm{O}$.
A matrix $\bm{A} \in \calS^n$ is called \emph{copositive} if $\bm{x}^\top\bm{A}\bm{x} \ge 0$ holds for all $\bm{x} \in \bbR_+^n$.
The cone of copositive matrices, the \emph{copositive cone} for short, in $\calS^n$ is denoted by $\COP^n$.

\section{An inexact subgradient algorithm for convex semi-infinite programming problems}\label{sec:SIP}
A copositive programming problem can be regarded as a convex semi-infinite programming problem.
In this section, we propose an inexact subgradient algorithm for general convex semi-infinite programming problems and establish a non-asymptotic convergence guarantee.

The semi-infinite programming problem considered here is of the form
\begin{equation}
\begin{alignedat}{3}
&\minimize_{x\in V} && \quad f(x)\\
&\subjectto && \quad g(x;\bm{\delta}) \le 0\ \text{for all $\bm{\delta} \in \Delta$},\\
&&& \quad x\in S.
\end{alignedat}\label{prob:SIP}
\end{equation}
Throughout this section, we make the following assumptions on the components of Problem~\eqref{prob:SIP}.

\begin{assumption} \label{asm:SIP}
\leavevmode
\begin{enumerate}[(a)]
\item $V$ is a finite-dimensional real vector space.
The space $V$ is equipped with an inner product $\langle \cdot,\cdot\rangle$, and we denote by $\norm{\cdot}$ the norm induced by the inner product.
\item $S$ is a non-empty closed convex subset of $V$.
\item The set of optimal solutions of \eqref{prob:SIP}, denoted by $S^*$, is non-empty.
\item $\Delta$ is a non-empty compact subset of $\bbR^n$.
\item $f\colon V \to (-\infty,+\infty]$ is a closed convex function that is subdifferentiable on $S$.
Moreover, there exists a positive constant $L_f$ such that $\lVert d\rVert \le L_f$ for all $x \in S$ and $d \in \partial f(x)$. \label{asm:f}
\item For every $\bm{\delta} \in \Delta$, $g(\cdot;\bm{\delta})\colon  V \to (-\infty,+\infty]$ is a closed convex function that is subdifferentiable on $S$.
Moreover, there exists a positive constant $L_g$ such that $\lVert d\rVert \le L_g$ for all $x \in S$, $\bm{\delta} \in \Delta$, and  $d \in \partial_x g(x;\bm{\delta})$. \label{asm:g_y}
\item For every $x \in S$, the inclusion $\Delta \subseteq \dom g(x;\cdot)$ holds and the function $g(x;\cdot)$ is continuous on $\Delta$. \label{asm:g_delta}
\end{enumerate}
For the constants $L_f$ and $L_g$ that appear in \eqref{asm:f} and \eqref{asm:g_y}, respectively, we define $L \coloneqq \max\{L_f,L_g\}$.
Furthermore, the optimal value of \eqref{prob:SIP} is denoted by $f^*$.
\end{assumption}

We define $G\colon V \to (-\infty,+\infty]$ as
\begin{equation}
G(x) \coloneqq \max_{\bm{\delta}\in \Delta}g(x;\bm{\delta}). \label{eq:G}
\end{equation}
Then we can recast the semi-infinite constraint in \eqref{prob:SIP}, requiring $g(x;\bm{\delta}) \le 0$ for all $\bm{\delta} \in \Delta$, as a single nonsmooth functional constraint $G(x) \le 0$.

\begin{remark}\label{rem:asm}
The latter claim in \eqref{asm:f} of Assumption~\ref{asm:SIP}, i.e., the boundedness of the subdifferential of $f$ at every $x\in S$, holds if $f$ is subdifferentiable and $L_f$-Lipschitz continuous on an open subset of $V$ that contains $S$.
Similarly, the boundedness of the subdifferential of $g(\cdot;\bm{\delta})$ at every $x\in S$ stated in \eqref{asm:g_y} holds if $g(\cdot;\bm{\delta})$ is subdifferentiable and $L_g$-Lipschitz continuous on an open subset of $V$ that contains $S$.
From the assumptions in \eqref{asm:g_y} and \eqref{asm:g_delta}, we see that $G$ is a closed convex function with $S \subseteq \dom G$.
Since $S$ is a closed convex set and since $f$ and $G$ are closed convex functions, $S^*$ is a closed convex set.
\end{remark}

\begin{algorithm}[t]
\caption{An inexact subgradient algorithm for convex semi-infinite programming problems}
\label{alg:proposed}
\begin{description}
\item[\bf Input:] Initial point $x_1 \in S$ and error tolerance $\epsilon > 0$
\item[\bf For $k = 1,2,3,\dots$:] Find $\bm{\delta}_k \in \Delta$ that approximately attains $G(x_k)$, i.e., with $g(x_k;\bm{\delta}_k) \simeq G(x_k)$.
\begin{enumerate}[(i)]
\item If $g(x_k;\bm{\delta}_k) \le \epsilon$, take $d_k \in \partial f(x_k)$ and let
\begin{equation*}
x_{k+1} \coloneqq P_S\left(x_k - \frac{\epsilon}{\lVert d_k\rVert^2}d_k\right).
\end{equation*}\label{enum:update_f}
\item Otherwise, namely if $g(x_k;\bm{\delta}_k) > \epsilon$, take $d_k \in \partial_x g(x_k;\bm{\delta}_k)$ and let
\begin{equation*}
x_{k+1} \coloneqq P_S\left(x_k - \frac{g(x_k;\bm{\delta}_k)}{\lVert d_k\rVert^2}d_k\right).
\end{equation*}\label{enum:update_g}
\end{enumerate}
\end{description}
\end{algorithm}

We present the proposed algorithm in Algorithm~\ref{alg:proposed}.
In principle, the per-iteration subproblem is to compute $G(x_k)$, the functional-constraint violation.
If the violation exceeds $\epsilon$, we update $x_k$ to reduce it; otherwise, we update $x_k$ to improve the objective.
However, computing $G(x_k)$ is generally nonconvex, so we seek $\bm{\delta}_k \in \Delta$ that makes $g(x_k;\bm{\delta}_k)$ approximate $G(x_k)$ well.
For example, we can use the sampling approach proposed by Wei, Haskell, and Zhao~\cite[Section~4]{WHZ2020} to approximate $G(x_k)$.
In Section~\ref{sec:COPP}, we discuss how to compute $G(x_k)$ exactly and inexactly when applying Algorithm~\ref{alg:proposed} to copositive programming problems.

We stop Algorithm~\ref{alg:proposed} after a finite number of iterations in practice.
If we stop after the $N$th iteration, we define
\begin{equation}
I_N \coloneqq \{k\in [N] \mid g(x_k;\bm{\delta}_k) \le \epsilon\}, \label{eq:def_IN}
\end{equation}
and for a minimizer $k^*$ of $f(x_k)$ over $k \in I_N$, we output $x_{k^*}$ as the approximate solution of \eqref{prob:SIP}.
Lemma~\ref{lem:convergence_SIP} provides an iteration-complexity bound ensuring that the approximate solution is well defined (i.e., the set $I_N$ is non-empty) and that it attains a prescribed accuracy.
For ease of description, we define
\begin{equation*}
\xi_k \coloneqq G(x_k) - g(x_k;\bm{\delta}_k), 
\end{equation*}
which is nonnegative by the definition of $G$, for each $k$.

\begin{lemma}\label{lem:convergence_SIP}
Suppose that Assumption~\ref{asm:SIP} holds.
Let $\epsilon > 0$ and $N$ be a positive integer such that
\begin{equation}
N \ge \frac{L^2\dist(x_1,S^*)^2}{\epsilon^2}. \label{eq:N_lb}
\end{equation}
Then the set $I_N$ defined in \eqref{eq:def_IN} is non-empty, and for a minimizer $k^*$ of $f(x_k)$ over $k \in I_N$, the following two inequalities hold:
\begin{align*}
f(x_{k^*}) &\le f^* + \epsilon,\\
G(x_{k^*}) &\le \epsilon + \xi_{k^*}.
\end{align*}

\end{lemma}

\begin{proof}
The set $S^*$ of optimal solutions is closed and convex (see Remark~\ref{rem:asm}).
For convenience, let $x^* \coloneqq P_{S^*}(x_1)$ and $r_k \coloneqq \lVert x_k - x^*\rVert$ for $k = 1,\dots,N+1$.

To derive a contradiction, we assume that $f(x_k) > f^* + \epsilon$ for all $k \in I_N$.
For every $k \in I_N$, we have
\begin{align*}
r_{k+1}^2 &\overset{\scriptsize \text{(a)}}= \left\lVert P_S\left(x_k - \frac{\epsilon}{\lVert d_k\rVert^2}d_k\right) - x^* \right\rVert^2\\
&\overset{\scriptsize \text{(b)}}\le \left\lVert x_k - \frac{\epsilon}{\lVert d_k\rVert^2}d_k - x^*\right\rVert^2\\
&= r_k^2 + \frac{2\epsilon}{\lVert d_k\rVert^2}\langle d_k,x^* - x_k\rangle + \frac{\epsilon^2}{\lVert d_k\rVert^2}\\
&\overset{\scriptsize \text{(c)}}\le r_k^2 + \frac{2\epsilon}{\lVert d_k\rVert^2}(f^* - f(x_k)) + \frac{\epsilon^2}{\lVert d_k\rVert^2}\\
&\overset{\scriptsize \text{(d)}}< r_k^2 - \frac{\epsilon^2}{\lVert d_k\rVert^2}\\
&\overset{\scriptsize \text{(e)}}\le r_k^2 - \frac{\epsilon^2}{L^2},
\end{align*}
where (a) follows from the update rule in \eqref{enum:update_f} of Algorithm~\ref{alg:proposed}, (b) holds by $x^* \in S$ and the nonexpansiveness of the projection $P_S$, (c) follows from $f^* = f(x^*)$ and $d_k \in \partial f(x_k)$, (d) is a consequence of the assumption for contradiction, and (e) follows from $\lVert d_k\rVert \le L$.
In addition, for every $k \in [N] \setminus I_N$, we have
\begin{align*}
r_{k+1}^2 &\overset{\scriptsize \text{(a)}}=  \left\lVert P_S\left(x_k - \frac{g(x_k;\bm{\delta}_k)}{\lVert d_k\rVert^2}d_k\right) - x^*\right\rVert^2\\
&\le \left\lVert x_k - \frac{g(x_k;\bm{\delta}_k)}{\lVert d_k\rVert^2}d_k - x^*\right\rVert^2\\
&= r_k^2 + \frac{2g(x_k;\bm{\delta}_k)}{\lVert d_k\rVert^2}\langle d_k,x^* - x_k\rangle + \frac{g(x_k;\bm{\delta}_k)^2}{\lVert d_k\rVert^2}\\
&\overset{\scriptsize \text{(b)}}\le r_k^2 + \frac{2g(x_k;\bm{\delta}_k)}{\lVert d_k\rVert^2}(g(x^*;\bm{\delta}_k) - g(x_k;\bm{\delta}_k)) + \frac{g(x_k;\bm{\delta}_k)^2}{\lVert d_k\rVert^2}\\
&\overset{\scriptsize \text{(c)}}< r_k^2 - \frac{\epsilon^2}{\lVert d_k\rVert^2}\\
&\overset{\scriptsize \text{(d)}}\le r_k^2 - \frac{\epsilon^2}{L^2},
\end{align*}
where (a) follows from the update rule in \eqref{enum:update_g} of Algorithm~\ref{alg:proposed}, (b) follows from $d_k \in \partial_x g(x_k;\bm{\delta}_k)$, (c) holds because $g(x_k;\bm{\delta}_k) > \epsilon$ and $g(x^*;\bm{\delta}_k) \le 0$, and (d) follows from $\lVert d_k\rVert \le L$.
Therefore,
\begin{equation}
r_{k+1}^2 < r_k^2 - \frac{\epsilon^2}{L^2} \label{eq:rk_recursion}
\end{equation}
holds for every $k \in [N]$.
Summing Inequality~\eqref{eq:rk_recursion} over $k$ from $1$ to $N$, we have
\begin{equation}
r_{N+1}^2 < r_1^2 - \frac{\epsilon^2 N}{L^2}. \label{eq:rk_recursion_sum}
\end{equation}
In particular, the right-hand side of \eqref{eq:rk_recursion_sum} is positive since $r_{N+1}^2$ is nonnegative.
On the other hand, since $N$ satisfies \eqref{eq:N_lb}, we have
\begin{equation*}
r_1^2 - \frac{\epsilon^2 N}{L^2} \le r_1^2 - \dist(x_1,S^*)^2 = 0,
\end{equation*}
which is a contradiction.
Hence, there exists $k \in I_N$ such that $f(x_k) \le f^* + \epsilon$.
This implies that the set $I_N$ is non-empty and $f(x_{k^*}) \le f^* + \epsilon$.

In addition, we have
\begin{equation*}
G(x_{k^*}) = g(x_{k^*};\bm{\delta}_{k^*}) + \xi_{k^*} \le \epsilon + \xi_{k^*},
\end{equation*}
where the equality follows from the definition of $\xi_{k^*}$ and the inequality follows from $k^* \in I_N$.
This completes the proof.
\end{proof}

Typically, we introduce a nonnegative parameter $\alpha$ and find $\bm{\delta}_k \in \Delta$ such that the inequality
\begin{equation}
\xi_k \coloneqq G(x_k) - g(x_k;\bm{\delta}_k) \le \alpha\epsilon \label{eq:xi_ineq_alphaeps}
\end{equation}
holds at each iteration.
If we know a way to evaluate the exact value of $G(x)$, we can set $\alpha = 0$.
Setting $\alpha$ to a positive value allows us to compute $G(x)$ approximately.
Since $\xi_{k^*}$ is also bounded by $\alpha \epsilon$, Lemma~\ref{lem:convergence_SIP} directly implies the following theorem.

\begin{theorem}\label{thm:convergence_SIP}
Suppose that Assumption~\ref{asm:SIP} holds.
Let $\epsilon > 0$, $\alpha \ge 0$, and $N$ be a positive integer satisfying \eqref{eq:N_lb}. 
For each iteration indexed by $k$, we find $\bm{\delta}_k \in \Delta$ satisfying \eqref{eq:xi_ineq_alphaeps}.
Then the set $I_N$ defined in \eqref{eq:def_IN} is non-empty, and for a minimizer $k^*$ of $f(x_k)$ over $k \in I_N$, we have the following two inequalities:
\begin{align*}
f(x_{k^*}) &\le f^* + \epsilon,\\
G(x_{k^*}) &\le (1+\alpha)\epsilon.
\end{align*}
\end{theorem}

In the case where $\alpha = 0$, i.e., where we solve subproblems exactly, we can view Algorithm~\ref{alg:proposed} as the application of the subgradient algorithm proposed by Nesterov~\cite[Equation~(3.2.24)]{Nesterov2018} to Problem~\eqref{prob:SIP}.
The algorithm in \cite{Nesterov2018} is comparable to the algorithms proposed in \cite[Equation~(3.2.13)]{Nesterov2004} and \cite{BBG+2010}, both applicable to convex programming problems with a single nonsmooth functional constraint.
Based on the algorithm~\cite{BBG+2010}, Wei, Haskell, and Zhao~\cite{WHZ2020} proposed an algorithm for solving convex semi-infinite programming problems.
Similar to Algorithm~\ref{alg:proposed}, this algorithm allows one to solve the subproblem inexactly at each iteration.
When applied to Problem~\eqref{prob:SIP}, the four algorithms differ in which function ($f$ or $G$) they improve at each iteration, in the choice of the step size, and in how the output is computed.

Except for Algorithm~\ref{alg:proposed} and Nesterov's algorithm~\cite{Nesterov2018}, the other algorithms require the compactness of $S$ and use the value $D \coloneqq \max_{x,y\in S}\norm{x-y} < +\infty$, the diameter of $S$, in the step size.
On the other hand, as shown in Theorem~\ref{thm:convergence_SIP}, Algorithm~\ref{alg:proposed} and the algorithm proposed by Nesterov~\cite{Nesterov2018} do not assume the compactness of $S$ but instead use $\dist(x_1,S^*)$, the distance between the initial point and the set of optimal solutions, to derive the number of iterations required to obtain a solution within a prescribed accuracy.
Note that $\dist(x_1,S^*) \le D$ holds when $S$ is a compact set of diameter $D$.

\begin{table}[tbp]
\centering
\caption{This table summarizes the number of iterations required to obtain $\bar{x} \in S$ such that $f(\bar{x}) \le f^* + \epsilon$ and $G(\bar{x}) \le \epsilon$ when we solve Problem~\eqref{prob:SIP} using each algorithm.
Here, $S$ is assumed to be a compact set of diameter $D$.
Bregman distances are utilized in the algorithms of \cite{BBG+2010} and \cite{WHZ2020}, but we show the number of iterations when the Bregman distance is Euclidean.
For the algorithm in \cite{WHZ2020}, the inequality $G(\bar{x}) \le \epsilon$ holds in expectation.
In addition, the algorithm in \cite{WHZ2020} yields $\bar{x} \in S$ such that $f(\bar{x}) \le f^* + \frac{3}{7}\epsilon$, while Algorithm~\ref{alg:proposed} yields $\bar{x} \in S$ such that $f(\bar{x}) \le f^* + \frac{1}{1 + \alpha}\epsilon$.
The number of iterations in the algorithm of \cite{BBG+2010} is obtained by taking the limit as $\beta \to 1$, where $\beta < 1$ is a parameter appearing in \cite[Corollary~2.1]{BBG+2010}.}
\begin{tabular}{lll}
\hline
Exactness & Method                                             & \#Iter.                                                                                               \\
\hline
Exact     & \cite{Nesterov2004}               & $\displaystyle \frac{3}{2} + 3\frac{D^2L^2}{\epsilon^2}$ \\
          & \cite{BBG+2010}                   & $\displaystyle \frac{(1 + \log 2)^2}{2(2-\sqrt{2})^2} \frac{D^2L^2}{\epsilon^2}   \simeq  4.2\frac{D^2L^2}{\epsilon^2}$                                                                \\
          & \cite{Nesterov2018}               & $\displaystyle \frac{D^2L^2}{\epsilon^2}$                                                                               \\
Inexact   & \cite{WHZ2020}                    & $\displaystyle 98\frac{D^2L^2}{\epsilon^2}$                                                                            \\
          & Algorithm~\ref{alg:proposed} & $\displaystyle (1 + \alpha)^2\frac{D^2L^2}{\epsilon^2}$      \\
\hline
\end{tabular}
\label{tab:compare_num_iter}
\end{table}

Under the compactness of $S$, Table~\ref{tab:compare_num_iter} summarizes the number of iterations required to obtain $\bar{x} \in S$ such that $f(\bar{x}) \le f^* + \epsilon$ and $G(\bar{x}) \le \epsilon$ when solving \eqref{prob:SIP} with each algorithm.
If we set the violation measure $\alpha$ to a reasonable value, e.g., $\alpha = 1$ as in the numerical experiments conducted in Section~\ref{sec:experiment}, Algorithm~\ref{alg:proposed} requires fewer iterations than the inexact algorithm of Wei, Haskell, and Zhao~\cite{WHZ2020} and even the exact algorithm of Beck et al.~\cite{BBG+2010}.

Note that the algorithm proposed by Nesterov~\cite{Nesterov2004} requires neither the error tolerance $\epsilon$ nor the number $N$ of iterations in advance and admits an error bound with respect to the current iteration index $k$.
However, the other algorithms use the error tolerance or the number of iterations when running.

\section{Application to copositive programming problems}\label{sec:COPP}
Copositive programming is a special case of convex semi-infinite programming; thus we can apply Algorithm~\ref{alg:proposed} to copositive programming problems.
In Section~\ref{subsec:general_COPP}, we discuss how to implement the proposed algorithm to solve general copositive programming problems.
In Section~\ref{subsec:test_CP}, we apply the proposed algorithm to the problem of testing whether a matrix is completely positive, which can be formulated as a copositive programming problem with a ball constraint.
Using Theorem~\ref{thm:convergence_SIP}, we derive a sufficient condition for a matrix not to be completely positive.
In Section~\ref{subsec:symcone}, we discuss a potential approach to solving copositive programming problems involving copositive cones over symmetric cones.

\subsection{General copositive programming problems}\label{subsec:general_COPP}
The copositive programming problem considered here is of the form
\begin{equation}
\begin{alignedat}{3}
&\minimize_{\bm{x}} && \quad \bm{c}^\top\bm{x}\\
&\subjectto && \quad \bm{A}_0 + \sum_{i=1}^m x_i\bm{A}_i \in \COP^n,\\
&&& \quad \bm{x}\in S,
\end{alignedat}\label{prob:COPP}
\end{equation}
where $\bm{A}_0,\dots,\bm{A}_m\in \calS^n$, $\bm{c}\in \bbR^m$, and $S$ is a non-empty closed convex subset of $\bbR^m$.
Since $\bm{A} \in \calS^n$ belongs to $\COP^n$ if and only if $\bm{\delta}^\top\bm{A}\bm{\delta} \ge 0$ for all $\bm{\delta} \in \Delta^{n-1}$, we can reformulate \eqref{prob:COPP} in the form of \eqref{prob:SIP}.
Specifically, the optimal value of \eqref{prob:COPP} is equal to that of the following convex semi-infinite programming problem:
\begin{equation}
\begin{alignedat}{3}
&\minimize_{\bm{x}} && \quad f(\bm{x}) \coloneqq \bm{c}^\top\bm{x}\\
&\subjectto && \quad g(\bm{x};\bm{\delta}) \coloneqq \bm{\delta}^\top\left(-\bm{A}_0 - \sum_{i=1}^m x_i\bm{A}_i\right)\bm{\delta} \le 0\ \text{for all $\bm{\delta} \in \Delta^{n-1}$},\\
&&& \quad \bm{x} \in S.
\end{alignedat}\label{prob:COPP_SIP}
\end{equation}

To implement Algorithm~\ref{alg:proposed} in practice, we need to deal with the following issues.
The first issue is to compute the value $L = \max\{L_f,L_g\}$.
The second issue is to find $\bm{\delta} \in \Delta^{n-1}$ such that $G(\bm{x}) - g(\bm{x};\bm{\delta}) \le \alpha \epsilon$ for a given $\bm{x}\in S$, where $G(\bm{x})$ is defined analogously to \eqref{eq:G}.

First, we discuss how to estimate the value $L$.
The functions $f$ and $g$ are differentiable with respect to $\bm{x}$, so we have
\begin{align}
\partial f(\bm{x}) &= \{\bm{c}\}, \label{eq:partial_f}\\
\partial_{\bm{x}} g(\bm{x};\bm{\delta}) &= \{(-\bm{\delta}^\top\bm{A}_1\bm{\delta},\dots,-\bm{\delta}^\top\bm{A}_m\bm{\delta})^\top\} \label{eq:partial_g}
\end{align}
for each $\bm{x} \in \bbR^m$ and $\bm{\delta} \in \Delta^{n-1}$.
From \eqref{eq:partial_f}, we can set $L_f$ to $\lVert\bm{c}\rVert_2$.
In addition, from \eqref{eq:partial_g}, we have
\begin{align}
\lVert (-\bm{\delta}^\top\bm{A}_1\bm{\delta},\dots,-\bm{\delta}^\top\bm{A}_m\bm{\delta})^\top\rVert_2 &= \sqrt{\sum_{i=1}^m\left(\sum_{k,l=1}^n(\bm{A}_i)_{kl}\delta_k\delta_l\right)^2} \nonumber\\
&\le \sqrt{\sum_{i=1}^m\left(\max_{1\le k\le l\le n}\abs{(\bm{A}_i)_{kl}}\right)^2}, \label{eq:Lg}
\end{align}
where the last inequality follows from $\bm{\delta}\in \Delta^{n-1}$.
The right-hand side of \eqref{eq:Lg} does not depend on $\bm{\delta} \in \Delta^{n-1}$, so it can be set to $L_g$.
Thus, we can set
\begin{equation}
L = \max\left\{\lVert\bm{c}\rVert_2, \sqrt{\sum_{i=1}^m\left(\max_{1\le k\le l\le n} \lvert(\bm{A}_i)_{kl}\rvert\right)^2}\right\}. \label{eq:est_L}
\end{equation}

Next, we discuss how to find $\bm{\delta} \in \Delta^{n-1}$ such that $G(\bm{x}) - g(\bm{x};\bm{\delta}) \le \alpha \epsilon$.
We note that $G(\bm{x})$ is the negative of the optimal value of the standard quadratic programming problem
\begin{equation}
\gamma(\bm{Q}) \coloneqq \min_{\bm{\delta} \in \Delta^{n-1}} \bm{\delta}^\top \bm{Q}\bm{\delta} \label{eq:stQP}
\end{equation}
with the coefficient matrix $\bm{Q} \in \calS^n$ given by $\bm{A}_0 + \sum_{i=1}^m x_i\bm{A}_i$.
In the following subsubsections, we give an overview of several approaches to solving a standard quadratic programming problem of the form \eqref{eq:stQP}, both exactly (the case where the violation parameter $\alpha$ equals $0$) and inexactly (the case where $\alpha > 0$).
An exact method is presented in Section~\ref{subsec:MILP}.
Inexact methods are presented in Sections~\ref{subsec:Grid}, \ref{subsec:SimpUnif}, and \ref{subsec:GridUnif}.
Whereas the exact method requires external solvers to solve a mixed-integer linear programming problem, only a finite number of function evaluations are performed in the inexact methods.
In Sections~\ref{subsec:Grid}, \ref{subsec:SimpUnif}, and \ref{subsec:GridUnif}, we only consider the case where $\alpha = 1$.
This is without loss of generality because for general $\alpha > 0$, we regard $\alpha\epsilon$ as a new $\epsilon$.

We briefly summarize the three inexact methods introduced in the subsequent subsubsections.
The method presented in Section~\ref{subsec:Grid} is deterministic and was originally proposed as a polynomial-time approximation scheme by Bomze and de Klerk~\cite{Bd2002}.
The method in Section~\ref{subsec:SimpUnif} is a brute-force randomized approach.
As shown in Section~\ref{subsec:NE_stQP}, this method has limited practical value; however, we include it to motivate the method in Section~\ref{subsec:GridUnif}.
The method in Section~\ref{subsec:GridUnif} is a randomized approach that combines the ideas of Sections~\ref{subsec:Grid} and \ref{subsec:SimpUnif}.
When $n$ is large, this method is expected to require fewer function evaluations to approximate $\gamma(\bm{Q})$ than the method in Section~\ref{subsec:Grid}.

In the following methods, we frequently use a lower bound for $\gamma(\bm{Q})$.
The choice of the lower bound is arbitrary, and there are many alternatives that are summarized in the literature~\cite{BLT2008}.
In our study, because we need to calculate the lower bound per iteration, it might be preferable to adopt a computationally inexpensive one.
As such a bound, under the convention that $1/0 = \infty$, $a + \infty = \infty$ for any $a\in \bbR \cup \{\infty\}$, and $1/\infty = 0$, Bomze, Locatelli, and Tardella~\cite{BLT2008} provided the following lower bound for $\gamma(\bm{Q})$:
\begin{equation*}
\underline{\gamma}(\bm{Q}) \coloneqq \min_{1\le i\le j\le n}Q_{ij} + \frac{1}{\displaystyle \sum_{k=1}^n\frac{1}{Q_{kk} - {\displaystyle \min_{1\le i\le j\le n}Q_{ij}}}}.
\end{equation*}
They showed that this bound outperforms other simple closed-form lower bounds reported in \cite[Section~2]{BLT2008}.

Moreover, we also use a Lipschitz constant of the function $\bm{\delta}^\top\bm{Q}\bm{\delta}$ on $\Delta^{n-1}$.
An estimate of the constant is as follows:
\begin{equation}
K(\bm{Q}) \coloneqq \max_{\bm{\delta} \in \Delta^{n-1}}\norm{\nabla (\bm{\delta}^\top\bm{Q}\bm{\delta})}_2 = 2\max_{\bm{\delta} \in \Delta^{n-1}}\norm{\bm{Q}\bm{\delta}}_2. \label{eq:K(Q)}
\end{equation}
Since $\norm{\bm{Q}\bm{\delta}}_2$ is convex and $\Delta^{n-1}$ is compact, the maximization problem in \eqref{eq:K(Q)} attains its maximum at an extreme point of $\Delta^{n-1}$, i.e., at one of $\bm{e}_1,\dots,\bm{e}_n$.
Letting $\bm{q}_i$ be the $i$th column of the matrix $\bm{Q}$ for each $i \in [n]$, it follows from \eqref{eq:K(Q)} that
\begin{equation*}
K(\bm{Q}) = 2\max_{1\le i\le n}\norm{\bm{q}_i}_2.
\end{equation*}
If $\bm{Q}=\bm{O}$, then the optimal value $\gamma(\bm{O})$ of \eqref{eq:stQP} is $0$, and the set of optimal solutions is $\Delta^{n-1}$.
Therefore, we may assume that $\bm{Q}\neq \bm{O}$.
Under this assumption, it follows that $K(\bm{Q})>0$.

\subsubsection{Exact method through mixed-integer linear programming}\label{subsec:MILP}
To solve the standard quadratic programming problem in \eqref{eq:stQP} exactly, we can utilize the mixed-integer linear programming reformulations proposed by Gondzio and Y{\i}ld{\i}r{\i}m~\cite{GY2021}.
In their paper, various mixed-integer linear programming reformulations of standard quadratic programming problems are presented.
Among them, we adopt the following reformulation based on \cite[Proposition~2]{GY2021}:\footnote{Using their terminology, we adopt formulation (MILP2) with the lower bound $\ell_1(\bm{Q})$.
They claim that this reformulation is suitable for large-scale instances because of its robust practical performance~\cite[page~319]{GY2021}.}
\begin{equation}
\begin{alignedat}{3}
&\minimize_{\bm{\delta},\bm{y},\bm{z},v} && \quad v\\
&\subjectto && \quad \bm{e}_j^\top\bm{Q}\bm{\delta} \le v + z_j \text{ for all $j \in [n]$},\\
&&& \quad \sum_{i=1}^n \delta_i = 1,\\
&&& \quad \delta_j \le y_j\text{ for all $j \in [n]$},\\
&&& \quad z_j \le \left(\max_{1\le i \le n}Q_{ij} - \underline{\gamma}(\bm{Q})\right)(1-y_j)\text{ for all $j \in [n]$}, \\
&&& \quad \bm{\delta} \in \bbR_+^n,\\
&&& \quad \bm{z} \in \bbR_+^n,\\
&&& \quad y_j \in \{0,1\}\text{ for all $j \in [n]$}.
\end{alignedat}\label{prob:stQP_MILP}
\end{equation}
For an optimal solution $(\bm{\delta}^*,\bm{y}^*,\bm{z}^*,v^*)$ of \eqref{prob:stQP_MILP}, the vector $\bm{\delta}^*$ represents an optimal solution of \eqref{eq:stQP} and $v^*$ represents its optimal value $\gamma(\bm{Q})$.

Note that finding an optimal solution of the standard quadratic programming problem in \eqref{eq:stQP} is more demanding than merely testing copositivity of the matrix $\bm{Q}$.
For the latter task, Anstreicher~\cite{Anstreicher2021} proposed a one-shot mixed-integer linear programming formulation for testing copositivity of a symmetric matrix.
This is a natural choice when the main goal is to decide whether a specific matrix is copositive.
We also note that copositivity can alternatively be tested through a mixed-integer linear programming reformulation equivalent to a standard quadratic programming problem, as presented in this subsubsection.
For a comparison between the two types of formulations for testing copositivity, see \cite{Peng2022}.
In contrast, our focus is on solving a general copositive programming problem in \eqref{prob:COPP} through an inexact subgradient framework.
In our approach, a mixed-integer linear programming formulation is used only as an exact option for solving a standard quadratic programming problem arising at each iteration.

To illustrate this distinction, consider the feasibility problem of finding $\bm{x}\in\bbR^m$ such that an affine linear mapping $\bm{A}(\bm{x}) \in \calS^n$ is copositive.
By applying the proposed framework, we can find a solution to this feasibility problem directly.
In contrast, to solve the same feasibility problem using the mixed-integer linear programming formulation of Anstreicher~\cite{Anstreicher2021}, one would need to embed the copositivity test into an outer procedure such as the ellipsoid method~\cite{GLS1993}.

\subsubsection{Inexact deterministic method through a regular grid of the standard simplex}\label{subsec:Grid}
The first inexact method for solving the standard quadratic programming problem is to discretize the standard simplex $\Delta^{n-1}$ by the regular grid defined as
\begin{equation}
\Delta_r^{n-1} \coloneqq \{\bm{\delta} \in \Delta^{n-1} \mid \text{Every element of $r\bm{\delta}$ is a nonnegative integer}\} \label{eq:regular_grid}
\end{equation}
for a positive integer $r$.
From \cite[Theorem~3.2]{Bd2002} (see also \cite[Theorem~2]{Nesterov2003}), we have
\begin{equation}
\min_{\bm{\delta} \in \Delta_r^{n-1}}\bm{\delta}^\top \bm{Q}\bm{\delta} - \gamma(\bm{Q}) \le \frac{1}{r}\left(\max_{1\le i\le n}Q_{ii} - \underline{\gamma}(\bm{Q})\right). \label{eq:Bd2002_ub}
\end{equation}
In other words, for a given $\epsilon > 0$, if $r$ satisfies
\begin{equation}
r \ge  \frac{1}{\epsilon}\left(\max_{1\le i\le n}Q_{ii} - \underline{\gamma}(\bm{Q})\right), \label{eq:Grid_r}
\end{equation}
it follows from \eqref{eq:Bd2002_ub} that
\begin{equation}
\min_{\bm{\delta} \in \Delta_r^{n-1}}\bm{\delta}^\top \bm{Q}\bm{\delta} - \gamma(\bm{Q}) \le \epsilon. \label{eq:Bd2002_ub_eps}
\end{equation}

Since $\Delta_r^{n-1}$ is a finite set regardless of $r$, only a finite number of function evaluations is required to solve $\min_{\bm{\delta} \in \Delta_r^{n-1}}\bm{\delta}^\top \bm{Q}\bm{\delta}$.
The number of evaluations is bounded by
\begin{equation*}
\lvert\Delta_r^{n-1}\rvert = \binom{n+r-1}{r} \le n^r,
\end{equation*}
which is polynomial in $n$ if we regard $r$ as a constant.
If the lower bound $\underline{\gamma}(\bm{Q})$ is weak, the number $\lvert \Delta_r^{n-1} \rvert$ of evaluations can become prohibitively large.
In such cases, rather than using the closed-form bound, it may be preferable to use the bounds obtained by solving semidefinite programming problems as presented in \cite{BLT2008}.

\subsubsection{Inexact randomized method through uniform sampling from the standard simplex}\label{subsec:SimpUnif}
The second inexact method for solving the standard quadratic programming problem is to discretize the standard simplex $\Delta^{n-1}$ by uniformly sampling from it.
Let $\bm{\delta}_1,\dots,\bm{\delta}_M$ be a collection of independent and identically distributed random vectors uniformly distributed on $\Delta^{n-1}$.
In Proposition~\ref{prop:Unif_829} shown below, we provide a lower bound for the number $M$ of samples sufficient to ensure that
\begin{equation}
\min_{1\le i\le M}\bm{\delta}_i^\top \bm{Q}\bm{\delta}_i - \gamma(\bm{Q}) \le \epsilon \label{eq:SimpUnif_approxsol}
\end{equation}
holds with probability at least $1-\phi$ for a given $\phi\in (0,1)$.
Throughout this subsubsection, $P$ denotes the uniform distribution on $\Delta^{n-1}$.

\begin{lemma}\label{lem:prob_ball_cap_simp_lb}
Let $r\in (0,\sqrt{2}]$.
For any $\bm{x}\in \Delta^{n-1}$, we have
\begin{equation*}
P(B(\bm{x},r) \cap \Delta^{n-1}) \ge \left(\frac{r}{\sqrt{2}}\right)^{n-1}.
\end{equation*}
\end{lemma}

\begin{proof}
Let $\vol(C)$ denote the $(n-1)$-dimensional Lebesgue measure of an $(n-1)$-dimensional set $C$ in $\bbR^n$.
Then we have
\begin{equation*}
p(\bm{x}) \coloneqq P(B(\bm{x},r) \cap \Delta^{n-1}) = \frac{\vol(B(\bm{x},r) \cap \Delta^{n-1})}{\vol(\Delta^{n-1})}.
\end{equation*}

First, we show that
\begin{equation}
\min_{\bm{x}\in \Delta^{n-1}}p(\bm{x}) = p(\bm{e}_1). \label{eq:min_p}
\end{equation}
Since $p(\bm{x})$ is nonnegative, we have
\begin{equation*}
\argmin_{\bm{x} \in \Delta^{n-1}} p(\bm{x}) = \argmin_{\bm{x} \in \Delta^{n-1}} p(\bm{x})^{n-1} =  \argmin_{\bm{x} \in \Delta^{n-1}} \vol(B(\bm{x},r) \cap \Delta^{n-1})^{n-1}.
\end{equation*}
By \cite[Lemma~6.21]{Jerrum2003}, the function $\vol(B(\bm{x},r) \cap \Delta^{n-1})^{n-1}$ with respect to $\bm{x} \in \Delta^{n-1}$ is concave, so it achieves the minimum at an extreme point of $\Delta^{n-1}$.
By the symmetry of $\Delta^{n-1}$, it does so at $\bm{x} = \bm{e}_1$.

For each $i \in \{2,\dots,n\}$, we define
\begin{equation*}
\bm{e}_i(r) \coloneqq \left(1 - \frac{r}{\sqrt{2}}\right)\bm{e}_1 + \frac{r}{\sqrt{2}}\bm{e}_i.
\end{equation*}
The convex hull of the set $\{\bm{e}_1,\bm{e}_2(r),\dots,\bm{e}_n(r)\}$, denoted by $\Delta^{n-1}(r)$, is an $(n-1)$-dimensional simplex with edge length $r$.
Since $\bm{e}_1,\bm{e}_2(r),\dots,\bm{e}_n(r) \in B(\bm{e}_1,r) \cap \Delta^{n-1}$ and $B(\bm{e}_1,r) \cap \Delta^{n-1}$ is convex, the inclusion
\begin{equation}
\Delta^{n-1}(r) \subseteq B(\bm{e}_1,r) \cap \Delta^{n-1} \label{eq:Delta(r)_subseteq}
\end{equation}
holds.
Therefore,
\begin{align*}
p(\bm{x}) \ge p(\bm{e}_1) = \frac{\vol(B(\bm{e}_1,r) \cap \Delta^{n-1})}{\vol(\Delta^{n-1})} \ge \frac{\vol(\Delta^{n-1}(r))}{\vol(\Delta^{n-1})} = \left(\frac{r}{\sqrt{2}}\right)^{n-1},
\end{align*}
holds for any $\bm{x}\in \Delta^{n-1}$, where the first inequality follows from \eqref{eq:min_p}, the second inequality follows from \eqref{eq:Delta(r)_subseteq}, and the last equality follows from the fact that the volume of an $(n-1)$-dimensional simplex with edge length $l$ is given by
\begin{equation*}
\frac{\sqrt{n}}{(n-1)!}\left(\frac{l}{\sqrt{2}}\right)^{n-1},
\end{equation*}
so we obtain the desired result.
\end{proof}

Below, we present the main result of this subsubsection.
For notational convenience, we define
\begin{equation}
m(\rho,\phi) \coloneqq \frac{\log\phi}{\log(1-\rho)} \label{eq:def_m}
\end{equation}
for $\rho,\phi \in (0,1)$.
\begin{proposition}\label{prop:Unif_829}
Let $\bm{\delta}_1,\dots,\bm{\delta}_M$ be a collection of independent and identically distributed random vectors uniformly distributed on $\Delta^{n-1}$.
In addition, let $m$ denote the function defined in \eqref{eq:def_m}.
For $\epsilon \in (0,\sqrt{2}K(\bm{Q})]$ and $\phi \in (0,1)$, if the number $M$ of samples satisfies
\begin{equation}
M \ge m\left(\left(\frac{\epsilon}{\sqrt{2}K(\bm{Q})}\right)^{n-1},\phi\right), \label{eq:Unif_829_num_samples}
\end{equation}
then, with probability at least $1 - \phi$, we have
\begin{equation*}
\min_{1\le i\le M}\bm{\delta}_i^\top \bm{Q}\bm{\delta}_i - \gamma(\bm{Q}) \le \epsilon.
\end{equation*}
\end{proposition}

\begin{proof}
Let $\bm{\delta}^*$ be an optimal solution of $\min_{\bm{\delta} \in \Delta^{n-1}} \bm{\delta}^\top\bm{Q}\bm{\delta}$.
Then the probability that $\bm{\delta}_i \not\in B(\bm{\delta}^*,\epsilon/K(\bm{Q}))$ holds for all $i \in [M]$ is $\prod_{i=1}^M P(\bm{\delta}_i \not\in B(\bm{\delta}^*,\epsilon/K(\bm{Q})))$ and it is bounded by
\begin{align*}
\prod_{i=1}^M P\left(\bm{\delta}_i \not\in B\left(\bm{\delta}^*,\frac{\epsilon}{K(\bm{Q})}\right)\right) &= \prod_{i=1}^M \left\{1 - P\left(\bm{\delta}_i \in B\left(\bm{\delta}^*,\frac{\epsilon}{K(\bm{Q})}\right)\right)\right\}\\
&\le \left\{1 - \left(\frac{\epsilon}{\sqrt{2}K(\bm{Q})}\right)^{n-1}\right\}^M\\
&\le \phi,
\end{align*}
where we use Lemma~\ref{lem:prob_ball_cap_simp_lb} to derive the first inequality and the second inequality follows from \eqref{eq:Unif_829_num_samples}.
This implies that with probability at least $1-\phi$, there exists $j \in [M]$ such that $\bm{\delta}_j \in  B(\bm{\delta}^*,\epsilon/K(\bm{Q}))$.
Then we have
\begin{equation*}
\min_{1\le i\le M} \bm{\delta}_i^\top \bm{Q}\bm{\delta}_i - (\bm{\delta}^*)^\top \bm{Q}\bm{\delta}^* \le \bm{\delta}_j^\top \bm{Q}\bm{\delta}_j - (\bm{\delta}^*)^\top \bm{Q}\bm{\delta}^* \le K(\bm{Q})\norm{\bm{\delta}_j - \bm{\delta}^*}_2 \le \epsilon,
\end{equation*}
so we obtain the desired result.
\end{proof}

\begin{remark}
By using \cite[Proposition~4.3]{WHZ2020}, we can obtain a result similar to Proposition~\ref{prop:Unif_829}.
Let $\bar{\gamma}(\bm{Q})$ be such that $\max_{\bm{\delta} \in \Delta^{n-1}}\bm{\delta}^\top\bm{Q}\bm{\delta} \le  \bar{\gamma}(\bm{Q})$ and $\underline{\gamma}(\bm{Q}) < \bar{\gamma}(\bm{Q})$ hold, and let $\bm{\delta}_1,\dots,\bm{\delta}_M$ be a collection of independent and identically distributed random vectors uniformly distributed on $\Delta^{n-1}$.
By applying \cite[Proposition~4.3]{WHZ2020} to the standard quadratic programming problem $\gamma(\bm{Q})$, we see that for a given $\epsilon > 0$, if the number $M$ of samples satisfies
\begin{equation}
M \ge m\left(\left(\frac{\epsilon}{2\sqrt{2}K(\bm{Q})}\right)^{n-1},\frac{\epsilon}{2(\bar{\gamma}(\bm{Q}) - \underline{\gamma}(\bm{Q}))}\right), \label{eq:WHZ2020_num_samples}
\end{equation}
the expected value of $\min_{1\le i\le M} \bm{\delta}_i^\top\bm{Q}\bm{\delta}_i$ is less than or equal to $\gamma(\bm{Q}) + \epsilon$.

We note that the number of samples shown in \eqref{eq:Unif_829_num_samples} and that in \eqref{eq:WHZ2020_num_samples} are exponential in $n$.
If $\rho$ is sufficiently small, then we have
\begin{equation}
m(\rho,\phi) \simeq \frac{1}{\rho}\log\left(\frac{1}{\phi}\right). \label{eq:M_simeq}
\end{equation}
The error tolerance $\epsilon$ is assumed to be small in practice, thus we obtain
\begin{align}
\eqref{eq:Unif_829_num_samples} &\simeq \left(\frac{\sqrt{2}K(\bm{Q})}{\epsilon}\right)^{n-1}\log\left(\frac{1}{\phi}\right), \label{eq:Unif_829_num_samples_approx}\\
\eqref{eq:WHZ2020_num_samples} &\simeq \left(\frac{2\sqrt{2}K(\bm{Q})}{\epsilon}\right)^{n-1}\log\left(\frac{2(\bar{\gamma}(\bm{Q}) - \underline{\gamma}(\bm{Q}))}{\epsilon}\right). \nonumber
\end{align}
\end{remark}

\subsubsection{Inexact randomized method through uniform sampling from a regular grid of the standard simplex}\label{subsec:GridUnif}
Generally, randomized approaches are used when the scale of problems is large and deterministic approaches cannot handle them.
However, we see from Proposition~\ref{prop:Unif_829} and \eqref{eq:Unif_829_num_samples_approx} that the randomized method in Section~\ref{subsec:SimpUnif} requires an exponential number of samples with respect to $n$ to obtain an $\epsilon$-approximate solution.
In this subsubsection, we discretize the standard simplex by uniformly sampling from its regular grid rather than from the standard simplex.
In Proposition~\ref{prop:GridUnif} shown below, we provide a probabilistic error bound on the optimal value, in which the number of function evaluations is expected to be less than the number of points in the regular grid if $n$ is large.
Recall that a regular grid $\Delta_r^{n-1}$ of the standard simplex $\Delta^{n-1}$ is defined as \eqref{eq:regular_grid}.
For simplicity, for $\bm{\delta} \in \Delta_r^{n-1}$, we define
\begin{equation*}
G_r(\bm{\delta}) \coloneqq B\left(\bm{\delta},\frac{\sqrt{2}}{r}\right) \cap \Delta_r^{n-1}.
\end{equation*}

\begin{lemma}\label{lem:B_lb}
For any $\bm{\delta} \in \Delta_r^{n-1}$, we have $\abs{G_r(\bm{\delta})} \ge n$.
In particular, the equality holds if $\bm{\delta}$ is any of $\bm{e}_1,\dots,\bm{e}_n$.
\end{lemma}

\begin{proof}
The vector $\bm{\delta}$ has at least one positive element.
In addition, every element of $\bm{\delta}$ is one of $0,1/r,2/r,\dots,1-1/r,1$.
Therefore, there exists $i \in [n]$ such that $\delta_i \ge 1/r$.
We prove $\abs{G_r(\bm{\delta})} \ge n$ by showing that
\begin{equation}
\left\{\bm{\delta} - \frac{1}{r}\bm{e}_i + \frac{1}{r}\bm{e}_j \relmiddle| j \in [n] \right\} \subseteq G_r(\bm{\delta}). \label{eq:B_elements}
\end{equation}
For every $j \in [n]$, we define $\bm{\delta}{(j)} \coloneqq \bm{\delta} - \frac{1}{r}\bm{e}_i + \frac{1}{r}\bm{e}_j \in \Delta_r^{n-1}$.
When $j = i$, we have $\bm{\delta}{(i)} = \bm{\delta}$, so $\bm{\delta}{(i)} \in G_r(\bm{\delta})$ holds.
When $j \neq i$, we have $\norm{\bm{\delta}(j) - \bm{\delta}}_2 = \sqrt{2}/r$, so $\bm{\delta}{(j)} \in G_r(\bm{\delta})$ holds.
Thus, we obtain $\abs{G_r(\bm{\delta})} \ge n$.

In particular, the converse inclusion in \eqref{eq:B_elements} holds when $\bm{\delta} = \bm{e}_1$.
Note that the index $i$ that satisfies $\delta_i \ge 1/r$ has to be $1$, and $\bm{\delta}{(j)} = (1-\frac{1}{r})\bm{e}_1 + \frac{1}{r}\bm{e}_j$ holds.
Let $\hat{\bm{\delta}} \in G_r(\bm{e}_1)$.
If $\hat{\delta}_1 \le 1 - 2/r$, we have
\begin{equation*}
\norm{\hat{\bm{\delta}} - \bm{e}_1}_2 \ge \abs{\hat{\delta}_1 - 1} \ge \frac{2}{r} > \frac{\sqrt{2}}{r},
\end{equation*}
which does not happen since $\hat{\bm{\delta}} \in B(\bm{e}_1,\sqrt{2}/r)$.
If $\hat{\delta}_1 = 1 - 1/r$, there exists $j \in \{2,\dots,n\}$ such that $\hat{\delta}_j = 1/r$ and $\hat{\delta}_k = 0$ for all $k \in \{2,\dots,n\} \setminus \{j\}$.
Then we have $\hat{\bm{\delta}} = \bm{\delta}{(j)}$.
If $\hat{\delta}_1 = 1$, it follows that $\hat{\bm{\delta}} = \bm{e}_1$, so $\hat{\bm{\delta}} = \bm{\delta}{(1)}$ holds.
Therefore, the set $G_r(\bm{e}_1)$ does not contain the elements other than $\bm{\delta}(1),\dots,\bm{\delta}(n)$ and we obtain $\abs{G_r(\bm{e}_1)} = n$.
By symmetry, $\abs{G_r(\bm{\delta})} = n$ also holds when $\bm{\delta}$ is any of $\bm{e}_1,\dots,\bm{e}_n$.
\end{proof}

\begin{proposition}\label{prop:GridUnif}
For a positive integer $r$, let $\bm{\delta}_1,\dots,\bm{\delta}_M$ be a collection of independent and identically distributed random vectors uniformly distributed on $\Delta_r^{n-1}$.
For $\phi \in (0,1)$, if the number $M$ of samples satisfies
\begin{equation}
M \ge m\left(\frac{n}{\abs{\Delta_r^{n-1}}},\phi\right), \label{eq:num_samples_ineq}
\end{equation}
then, with probability at least $1 - \phi$, we have
\begin{equation}
\min_{1\le i\le M} \bm{\delta}_i^\top \bm{Q} \bm{\delta}_i - \gamma(\bm{Q}) \le \frac{1}{r}\left(\sqrt{2}K(\bm{Q}) + \max_{1\le i\le n}Q_{ii} - \underline{\gamma}(\bm{Q})
\right). \label{eq:GridUnif_bound}
\end{equation}
In particular, for a given $\epsilon > 0$, if $r$ satisfies
\begin{equation}
r \ge \frac{1}{\epsilon}\left(\max_{1\le i\le n}Q_{ii} - \underline{\gamma}(\bm{Q})\right), \label{eq:GUnif_r}
\end{equation}
then, with probability at least $1-\phi$, we have
\begin{equation*}
\min_{1\le i\le M} \bm{\delta}_i^\top \bm{Q} \bm{\delta}_i - \gamma(\bm{Q})
\le \left(\frac{\sqrt{2}K(\bm{Q})}{\displaystyle \max_{1\le i\le n}Q_{ii} - \underline{\gamma}(\bm{Q})} + 1\right)\epsilon.
\end{equation*}
\end{proposition}

\begin{proof}
Let $\bm{\delta}^*$ be an optimal solution of $\min_{\bm{\delta} \in \Delta_r^{n-1}} \bm{\delta}^\top \bm{Q}\bm{\delta}$.
Then the probability that $\bm{\delta}_i \not\in B(\bm{\delta}^*,\sqrt{2}/r)$ holds for all $i \in [M]$ is $(1 - \abs{G_r(\bm{\delta}^*)} / \abs{\Delta_r^{n-1}})^M$ and it is bounded by
\begin{equation*}
\left(1 - \frac{\abs{G_r(\bm{\delta}^*)}}{\abs{\Delta_r^{n-1}}}\right)^M \le \left(1 - \frac{n}{\abs{\Delta_r^{n-1}}}\right)^M \le \phi,
\end{equation*}
where we use Lemma~\ref{lem:B_lb} to derive the first inequality and the second inequality follows from \eqref{eq:num_samples_ineq}.
This implies that with probability at least $1-\phi$, there exists $j \in [M]$ such that $\bm{\delta}_j \in G_r(\bm{\delta}^*)$.
Then we have
\begin{align*}
\min_{1\le i\le M} \bm{\delta}_i^\top \bm{Q} \bm{\delta}_i - \gamma(\bm{Q}) &\le (\bm{\delta}_j^\top \bm{Q} \bm{\delta}_j - (\bm{\delta}^*)^\top \bm{Q} \bm{\delta}^*) + ((\bm{\delta}^*)^\top \bm{Q} \bm{\delta}^* - \gamma(\bm{Q}))\\
&\le K(\bm{Q})\norm{\bm{\delta}_j - \bm{\delta}^*}_2 + \frac{1}{r}\left(\max_{1\le i\le n}Q_{ii} - \underline{\gamma}(\bm{Q})\right)\\
&\le \frac{1}{r}\left(\sqrt{2}K(\bm{Q}) + \max_{1\le i\le n}Q_{ii} - \underline{\gamma}(\bm{Q})\right),
\end{align*}
where we use \eqref{eq:Bd2002_ub} to derive the second inequality and we use $\bm{\delta}_j \in G_r(\bm{\delta}^*)$ to derive the third inequality.
\end{proof}

Proposition~\ref{prop:GridUnif} suggests the possibility that we can compute an approximate value of $\gamma(\bm{Q})$ with a smaller number of function evaluations than that required by the method in Section~\ref{subsec:Grid}.
Using the approximation shown in \eqref{eq:M_simeq}, we see that
\begin{equation*}
\eqref{eq:num_samples_ineq} \simeq \frac{1}{n}\log\left(\frac{1}{\phi}\right) \abs{\Delta_r^{n-1}}.
\end{equation*}
The value $\frac{1}{n}\log(\frac{1}{\phi})$ is expected to be less than $1$ if $n$ is large.
In such a case, with a smaller number of function evaluations than that in Section~\ref{subsec:Grid}, we can derive an approximation of $\gamma(\bm{Q})$ that enjoys the theoretical guarantee shown in \eqref{eq:GridUnif_bound}, although this bound is weaker than \eqref{eq:Bd2002_ub}.

\subsection{Testing complete positivity of a matrix}\label{subsec:test_CP}
A matrix $\bm{A}\in \calS^n$ is said to be \emph{completely positive} if there exist a positive integer $k$ and $\bm{a}_1,\dots,\bm{a}_k \in \bbR_+^n$ such that $\bm{A} = \sum_{i=1}^k \bm{a}_i\bm{a}_i^\top$.
For a matrix $\bm{C} \in \calS^n$, the problem of testing the complete positivity of $\bm{C}$ can be formulated as the following copositive programming problem with a ball constraint:
\begin{equation}
\begin{alignedat}{3}
&\minimize_{\bm{X}} && \quad \langle\bm{C},\bm{X}\rangle \\
&\subjectto && \quad \bm{X} \in \COP^n,\\
&&& \quad \bm{X} \in B(\bm{O},1).
\end{alignedat}\label{prob:test_CP}
\end{equation}
Without loss of generality, we may assume that $\norm{\bm{C}}_{\rm F} = 1$.
Since $\bm{X} = \bm{O}$ is a feasible solution of this problem, its optimal value is at most $0$.
In addition, by the duality between completely positive cones and copositive cones~\cite[Theorem~2.1]{HN1963}, the matrix $\bm{C}$ is completely positive if and only if the optimal value is $0$.

We can reformulate \eqref{prob:test_CP} in the form of \eqref{prob:SIP}, namely,
\begin{equation}
\begin{alignedat}{3}
&\minimize_{\bm{X}} && \quad f(\bm{X}) \coloneqq  \langle\bm{C},\bm{X}\rangle \\
&\subjectto && \quad g(\bm{X};\bm{\delta}) \coloneqq \langle \bm{X},-\bm{\delta}\bm{\delta}^\top\rangle \le 0\ \text{for all $\bm{\delta} \in \Delta^{n-1}$},\\
&&& \quad \bm{X} \in B(\bm{O},1).
\end{alignedat}\label{prob:test_CP_SIP}
\end{equation}
Therefore, we can directly apply Algorithm~\ref{alg:proposed} developed in Section~\ref{sec:SIP} to Problem~\eqref{prob:test_CP_SIP}.
As in Section~\ref{subsec:general_COPP}, the subproblem $G(\bm{X}) \coloneqq \max_{\bm{\delta}\in \Delta^{n-1}}g(\bm{X};\bm{\delta})$ that needs to be solved at each iteration is essentially the standard quadratic programming problem with coefficient matrix $\bm{X}$.
The calculation of the constant $L_f$ is the same as that in Section~\ref{subsec:general_COPP}; we can set $L_f = 1$ under the assumption that $\norm{\bm{C}}_{\rm F} = 1$.
In addition, we can set $L_g = 1$ since $\partial_{\bm{X}}g(\bm{X};\bm{\delta}) = \{-\bm{\delta}\bm{\delta}^\top\}$ and
\begin{equation*}
\norm{-\bm{\delta}\bm{\delta}^\top}_{\rm F} = \norm{\bm{\delta}}_2^2 \le \left(\sum_{i=1}^n \delta_i\right)^2 = 1
\end{equation*}
hold for any $\bm{\delta} \in \Delta^{n-1}$.
Then $L \coloneqq \max\{L_f,L_g\}$ equals $1$ and Theorem~\ref{thm:convergence_SIP} can be translated into the following proposition.

\begin{proposition}\label{prop:convergence_test_CP}
We apply Algorithm~\ref{alg:proposed} to \eqref{prob:test_CP_SIP}, in which we set the initial point $\bm{X}_1$ to $\bm{O}$.
Let $N$ be a positive integer such that
\begin{equation}
N \ge \frac{1}{\epsilon^2}. \label{eq:num_iter_test_CP}
\end{equation}
For each iteration indexed by $k$, we find $\bm{\delta}_k \in \Delta^{n-1}$ such that $G(\bm{X}_k) - g(\bm{X}_k;\bm{\delta}_k) \le \alpha\epsilon$.
Then the set
\begin{equation}
I_N \coloneqq \{ k \in [N] \mid g(\bm{X}_k;\bm{\delta}_k) \le \epsilon\} \label{eq:IN_test_CP}
\end{equation}
is non-empty and for a minimizer $k^*$ of $f(\bm{X}_k)$ over $k \in I_N$, the following two inequalities hold:
\begin{align*}
f(\bm{X}_{k^*}) &\le f^* + \epsilon,\\
G(\bm{X}_{k^*}) &\le (1+\alpha)\epsilon.
\end{align*}
\end{proposition}

\begin{proof}
Since $\bm{X}_1 = \bm{O}$ and the set $S^*$ of optimal solutions of Problem~\eqref{prob:test_CP_SIP} is included in $B(\bm{O},1)$, we have
\begin{equation}
\dist(\bm{X}_1,S^*) \le 1. \label{eq:test_CP_ub_dist}
\end{equation}
Substituting $L = 1$ and \eqref{eq:test_CP_ub_dist} into \eqref{eq:N_lb}, Theorem~\ref{thm:convergence_SIP} leads to the desired result.
\end{proof}

\begin{remark}
By vectorizing the matrices in \eqref{prob:test_CP}, we can reformulate Problem~\eqref{prob:test_CP} as a problem of the form \eqref{prob:COPP_SIP}, to which the discussion in Section~\ref{subsec:general_COPP} applies.
Specifically, letting $\bm{E}_i \coloneqq \bm{e}_i\bm{e}_i^\top$ for each $i = 1,\dots,n$ and $\bm{E}_{ij} \coloneqq (\bm{e}_i\bm{e}_j^\top + \bm{e}_j\bm{e}_i^\top) / \sqrt{2}$ for each $1\le i < j\le n$ and introducing variables
\begin{equation*}
\bm{x} = (x_{11},x_{12},x_{22},\dots,x_{1n},\dots,x_{nn})^\top \in \bbR^{\frac{n(n+1)}{2}},
\end{equation*}
we can reformulate Problem~\eqref{prob:test_CP} as
\begin{equation}
\begin{alignedat}{3}
&\minimize_{\bm{x}} && \quad \tilde{f}(\bm{x}) \coloneqq \sum_{i=1}^n C_{ii}x_{ii} + \sum_{1\le i < j\le n}\sqrt{2}C_{ij}x_{ij} \\
&\subjectto && \quad \tilde{g}(\bm{x};\bm{\delta}) \coloneqq \bm{\delta}^\top\left(-\sum_{i=1}^n x_{ii}\bm{E}_i - \sum_{1\le i < j \le n}x_{ij}\bm{E}_{ij}\right)\bm{\delta} \le 0\ \text{for all $\bm{\delta} \in \Delta^{n-1}$},\\
&&& \quad \bm{x} \in B(\bm{0},1).
\end{alignedat} \label{prob:test_CP_vecvar}
\end{equation}

However, the constant $L \coloneqq \max\{L_{\tilde{f}},L_{\tilde{g}}\}$ calculated as \eqref{eq:est_L} depends on $n$ unlike the formulation in \eqref{prob:test_CP_SIP}.
The coefficient vector, denoted by $\bm{c} \in \bbR^{n(n+1)/2}$, corresponding to the objective function in \eqref{prob:test_CP_vecvar} satisfies $\norm{\bm{c}}_2 = \norm{\bm{C}}_{\rm F} = 1$, so we have $L_{\tilde{f}} = 1$.
In addition, by the discussion in Section~\ref{subsec:general_COPP}, the constant $L_{\tilde{g}}$ can be set to
\begin{align*}
L_{\tilde{g}} = \sqrt{
\sum_{i=1}^n\left(\max_{1\le k \le l \le n}\abs{(\bm{E}_i)_{kl}} \right)^2 + \sum_{1\le i < j\le n}\left(\max_{1\le k \le l \le n}\abs{(\bm{E}_{ij})_{kl}}
\right)^2} = \frac{\sqrt{n^2 + 3n}}{2}.
\end{align*}
Therefore, we have $L = \sqrt{n^2 + 3n}/2$ and the iteration-complexity bound shown in \eqref{eq:N_lb} increases as $n$ rises.
\end{remark}

Since the approximate solution $\bm{X}_{k^*}$ may be an infeasible solution of Problem~\eqref{prob:test_CP}, the negativity of the approximate optimal value $\langle\bm{C},\bm{X}_{k^*}\rangle$ does not necessarily mean that $\bm{C}$ is not completely positive.
However, as shown in the following theorem, if the approximate optimal value is less than a threshold, $\bm{C}$ is guaranteed not to be completely positive.

\begin{theorem}\label{thm:suff_cond_not_CP}
We apply Algorithm~\ref{alg:proposed} to \eqref{prob:test_CP_SIP}, in which we set the initial point $\bm{X}_1$ to $\bm{O}$.
Let $N$ be a positive integer satisfying \eqref{eq:num_iter_test_CP}.
For each iteration indexed by $k$, we find $\bm{\delta}_k \in \Delta^{n-1}$ such that $G(\bm{X}_k) - g(\bm{X}_k;\bm{\delta}_k) \le \alpha\epsilon$.
After $N$ iterations of the algorithm, we let $k^* \in \argmin\{\langle \bm{C},\bm{X}_k\rangle \mid k \in I_N\}$, where $I_N$ is defined as \eqref{eq:IN_test_CP}.
If the inequality
\begin{equation}
\langle \bm{C},\bm{X}_{k^*}\rangle < -n(1+\alpha)\epsilon \label{eq:suff_cond_not_CP}
\end{equation}
holds, $\bm{C}$ is not completely positive.
\end{theorem}

\begin{proof}
By Proposition~\ref{prop:convergence_test_CP}, the inequality
\begin{equation}
G(\bm{X}_{k^*}) \le (1 + \alpha)\epsilon \label{eq:ineq_G_test_CP}
\end{equation}
holds.
Let $\bar{\bm{X}} \coloneqq P_{\COP^n}(\bm{X}_{k^*})$.
Then it follows that
\begin{equation*}
\norm{\bm{X}_{k^*} - \bar{\bm{X}}}_{\rm F} \le n\max\{G(\bm{X}_{k^*}),0\} \le n (1 + \alpha)\epsilon,
\end{equation*}
where we use \cite[Equation~(6.16)]{HS2010} to derive the first inequality and we use \eqref{eq:ineq_G_test_CP} to derive the second inequality.
Using this inequality, the Cauchy--Schwarz inequality, and $\norm{\bm{C}}_{\rm F} = 1$, we have
\begin{equation*}
\abs{\langle \bm{C},\bm{X}_{k^*}\rangle - \langle \bm{C},\bar{\bm{X}}\rangle} \le n (1 + \alpha)\epsilon.
\end{equation*}
Combining this with the assumption in \eqref{eq:suff_cond_not_CP}, we see that
\begin{equation*}
\langle \bm{C},\bar{\bm{X}}\rangle \le \langle \bm{C},\bm{X}_{k^*}\rangle + n (1 + \alpha)\epsilon < 0.
\end{equation*}
Since $\bar{\bm{X}} \in \COP^n$, this inequality implies that $\bm{C}$ is not completely positive.
\end{proof}

\begin{remark}\label{rem:set_eps}
We note that the inequality $\langle \bm{C},\bm{X}_{k^*} \rangle \ge -1$ holds by the Cauchy--Schwarz inequality, $\norm{\bm{C}}_{\rm F} = 1$, and the constraint $\bm{X}_{k^*} \in B(\bm{O},1)$.
This implies that $\epsilon$ must satisfy $\epsilon < \frac{1}{n(1 + \alpha)}$ to use the condition in \eqref{eq:suff_cond_not_CP}.
For example, we can set
\begin{equation*}
\epsilon \coloneqq \frac{1}{tn(1+\alpha)} \label{eq:eps_suff_cond_not_CP}
\end{equation*}
for a constant $t > 1$, in which case only
\begin{equation}
t^2(1+\alpha)^2 n^2 = O(n^2) \label{eq:num_iter_test_CP_c}
\end{equation}
iterations are necessary to test non-complete positivity by using Theorem~\ref{thm:suff_cond_not_CP}.
Note, however, that $t$ may need to be chosen large in order to detect non-complete positivity, in which case the number of iterations in \eqref{eq:num_iter_test_CP_c} also becomes large.
We revisit the choice of $t$ in Section~\ref{subsec:NE_CP}.
\end{remark}

\begin{remark}\label{rem:certify_nonCP}
The number of iterations shown in \eqref{eq:num_iter_test_CP} is necessary to guarantee that the set $I_N$ is non-empty.
However, in running the algorithm, we may find a matrix $\bm{X}_k$ satisfying $g(\bm{X}_k;\bm{\delta}_k) \le \epsilon$ even if $k < 1/\epsilon^2$.
If we obtain a matrix $\bm{X}_k$ such that $g(\bm{X}_k;\bm{\delta}_k) \le \epsilon$ and $\langle \bm{C},\bm{X}_{k}\rangle < -n(1+\alpha)\epsilon$, the matrix $\bm{X}_k$ provides a certificate that $\bm{C}$ is not completely positive.
\end{remark}

\subsection{The case of symmetric cones beyond nonnegative orthants}\label{subsec:symcone}
In the previous subsections, we discussed how to solve copositive programming problems constrained by the copositive cone $\COP^n$.
Recently, studies have been conducted on copositive cones determined by \emph{symmetric cones} (self-dual and homogeneous cones) beyond nonnegative orthants, and on their associated copositive programming problems~\cite{NL20XX,NL2025,NN2024_Approximation,NN2024_Generalizations,Orlitzky2021}.
Let $\bbK$ be a symmetric cone in a finite-dimensional real vector space $\bbV$ equipped with an inner product denoted by $\bullet$.
\emph{The copositive cone over the symmetric cone $\bbK$}, denoted by $\COP(\bbK)$, is the cone of self-adjoint linear transformations $A$ on $\bbV$ such that $x \bullet A(x) \ge 0$ for all $x \in \bbK$.
If $\bbK$ is the nonnegative orthant $\bbR_+^n$, the cone $\COP(\bbR_+^n)$ can be identified with $\COP^n$.

In principle, we can use Algorithm~\ref{alg:proposed} to solve copositive programming problems involving a copositive cone over a symmetric cone $\bbK$.
We let $e$ be an interior point in $\bbK$ and define
\begin{equation*}
\Delta_{\bbK,e} \coloneqq \{\delta \in \bbK \mid e\bullet \delta = 1\}.
\end{equation*}
The set $\Delta_{\bbK,e}$ can be regarded as a generalization of the standard simplex; indeed, by setting $\bbK$ to $\bbR_+^n$ and $e$ to the vector with all elements $1$, we have $\Delta_{\bbK,e} = \Delta^{n-1}$.
Generally, $\Delta_{\bbK,e}$ is a compact slice of $\bbK$, and $A \in \COP(\bbK)$ if and only if $\delta \bullet A(\delta) \ge 0$ for all $\delta\in \Delta_{\bbK,e}$.
The subproblem that we need to solve at each iteration of Algorithm~\ref{alg:proposed} is essentially
\begin{equation}
\min_{\delta \in \Delta_{\bbK,e}} \delta \bullet Q(\delta), \label{eq:gstQP}
\end{equation}
where $Q$ is a self-adjoint linear transformation on $\bbV$.
Problem~\eqref{eq:gstQP} corresponds to the standard quadratic programming problem in \eqref{eq:stQP}.
Whereas we can utilize the mixed-integer linear programming reformulation to solve standard quadratic programming problems globally, it is unknown how to obtain a global solution to \eqref{eq:gstQP}.

In what follows, we explore a potential approach to deriving a global solution to Problem~\eqref{eq:gstQP}.
It follows from \cite[Section~3]{FCC2020} that the optimal value of \eqref{eq:gstQP} is equal to that of the following copositive programming problem with a scalar variable:
\begin{equation*}
\max_{\lambda\in \bbR}\{ \lambda \mid Q - \lambda e\otimes e \in \COP(\bbK)\}.
\end{equation*}
An optimal solution of this problem, denoted by $\lambda^*$, can be found by combining membership testing for $\COP(\bbK)$ with the bisection method.
Despite the absence of established numerical methods for the membership problem of $\COP(\bbK)$, Orlitzky~\cite{Orlitzky2021} explores a potential extension of a recursive method to solve the membership problem for $\COP(\bbR_+^n)$ to that for $\COP(\bbK)$.
Since $Q - \lambda^* e\otimes e$ lies in the boundary of $\COP(\bbK)$, there exists $\delta^* \in \Delta_{\bbK,e}$ such that
\begin{equation*}
0 = \delta^* \bullet (Q - \lambda^* e\otimes e)(\delta^*) = \delta^* \bullet Q(\delta^*) - \lambda^*,
\end{equation*}
i.e., $\delta^* \bullet Q(\delta^*) = \lambda^*$.
Therefore, $\delta^*$ is an optimal solution of Problem~\eqref{eq:gstQP}.

\section{Numerical experiments} \label{sec:experiment}
In this section, we verify the effectiveness of the methods introduced so far through numerical experiments.
As outlined in Section~\ref{subsec:general_COPP}, four distinct methods for solving standard quadratic programming problems were presented.
We compare these approaches in Section~\ref{subsec:NE_stQP}.
In Section~\ref{subsec:NE_COPP}, we compare the performance of these methods when they are incorporated into Algorithm~\ref{alg:proposed} for solving copositive programming problems.
In Section~\ref{subsec:NE_CP}, we demonstrate that the approach introduced in Section~\ref{subsec:test_CP} is effective in detecting non-complete positivity of a matrix.

All experiments were conducted in MATLAB (R2025a) on a computer with an Apple M1 and \qty{16}{\giga\byte} memory.
The Gurobi solver~\cite{Gurobi} (version 12.0.2) was used to solve mixed-integer linear programming problems.

\subsection{Comparison among the methods for solving standard quadratic programming problems}\label{subsec:NE_stQP}
We introduced four methods to solve standard quadratic programming problems in Sections~\ref{subsec:MILP}, \ref{subsec:Grid}, \ref{subsec:SimpUnif}, and \ref{subsec:GridUnif}.
In this subsection, we compare them numerically in terms of solution accuracy and computational time.
The standard quadratic programming problem solved here is of the form \eqref{eq:stQP}.
For each $n\in \{5,10,50,100,500,1000\}$, we generated ten coefficient matrices $\bm{Q} \in \calS^n$ whose elements independently followed the uniform distribution on the interval $[-1,1]$.

First, we compare the methods presented in Sections~\ref{subsec:MILP}, \ref{subsec:Grid}, and \ref{subsec:SimpUnif} (denoted \textsf{MILP}, \textsf{Grid}, and \textsf{SUnif}, respectively).
We set the parameters in \textsf{Grid} and \textsf{SUnif} as follows.
In both methods, we set $\epsilon$ that appears in \eqref{eq:Grid_r} and \eqref{eq:Unif_829_num_samples} to $1$ or $0.1$.
When we used \textsf{Grid}, we set $r$ to the minimum positive integer satisfying \eqref{eq:Grid_r}.
When we used \textsf{SUnif}, we set the number $M$ of samples to the minimum integer satisfying \eqref{eq:Unif_829_num_samples} and set $\phi$ to $0.05$, so that \eqref{eq:SimpUnif_approxsol} holds with probability at least $0.95$.
For each $(n,\epsilon)$, we solved ten instances using each method and then took the average of the results.
The maximum execution time was set to \qty{3600}{\second}.

\begin{table}[tbp]
\centering
\small
\caption{The results of solving standard quadratic programming problems by using \textsf{MILP}, \textsf{Grid}, and \textsf{SUnif}.
The average objective values obtained by solving standard quadratic programming problems with each method are listed in the ``Obj. val.'' column and the average computational time required to solve them is listed in the ``Time [s]'' column.
The ``Dev. from \textsf{MILP}'' column denotes the difference between each average objective value and that obtained by using \textsf{MILP} for the same $n$.
The ``$M$'' column denotes the average number $M$ of points $\bm{\delta}_1,\dots,\bm{\delta}_M \in \Delta^{n-1}$ that are used to calculate an approximate optimal value $\min_{1\le i\le M} \bm{\delta}_i^\top \bm{Q}\bm{\delta}_i$ for a standard quadratic programming problem of the form \eqref{eq:stQP}.
The symbol $>$\texttt{realmax} in the ``$M$'' column indicates that the average number of points exceeds \texttt{realmax} in MATLAB, which is the largest finite floating-point number in IEEE double precision (approximately \num{1.7977e+308}).}
\begin{tabular}{
r
l
r
S[table-number-alignment=right,table-format=1.1e1]
S[table-number-alignment=right,table-format=+1.2e+1]
S[table-number-alignment=right,table-format=1.2e+1]
S[table-number-alignment=right,table-format=1.1e+1]}
\hline
\multicolumn{1}{l}{$n$}  & Method                             & \multicolumn{1}{l}{$\epsilon$} & \multicolumn{1}{l}{$M$}           & \multicolumn{1}{l}{Obj. val.}          & \multicolumn{1}{l}{Dev. from \textsf{MILP}} &  \multicolumn{1}{l}{Time [\unit{\second}]}       \\ \hline
5    & \textsf{MILP}     & \multicolumn{1}{r}{---}        & \multicolumn{1}{r}{---}                & -7.96E-01          & \multicolumn{1}{r}{---}                                      & 3.3E-03            \\
     & \textsf{Grid}     & 1          & 1.5E+01            & -7.90E-01          & 5.54E-03                                 & 2.9E-03            \\
     &                                    & 0.1        & 6.3E+03            & -7.96E-01          & 1.01E-04   & 1.6E-02            \\
     & \textsf{SUnif} & 1          & 1.2E+03            & -6.27E-01          & 1.69E-01                                 & 2.8E-02            \\
     &                                    & 0.1        & 1.2E+07            & -7.75E-01          & 2.04E-02                                 & 5.3E-01            \\ \hline
10   & \textsf{MILP}     & \multicolumn{1}{r}{---}        & \multicolumn{1}{r}{---}                & -8.69E-01          & \multicolumn{1}{r}{---}                                      & 7.7E-03            \\
     & \textsf{Grid}     & 1          & 5.5E+01            & -8.64E-01          & 5.00E-03                                 & 2.7E-03            \\
     &                                   & 0.1        & 3.3E+06            & -8.69E-01          & 6.00E-06                                 & 2.8E+00            \\
     & \textsf{SUnif} & 1          & 3.6E+08            & -6.84E-01          & 1.85E-01                                 & 3.3E+00            \\
     &                                    & 0.1        & 3.6E+16 & \multicolumn{1}{r}{---} & \multicolumn{1}{r}{---}                       & \multicolumn{1}{r}{$>$3600} \\ \hline
50   & \textsf{MILP}     & \multicolumn{1}{r}{---}        & \multicolumn{1}{r}{---}                & -9.68E-01          & \multicolumn{1}{r}{---}                                      & 2.6E-01            \\
     & \textsf{Grid}     & 1          & 1.3E+03            & -9.67E-01          & 9.66E-04                                 & 7.3E-03            \\
     &                                    & 0.1        & 9.9E+16 & \multicolumn{1}{r}{---} & \multicolumn{1}{r}{---}                       & \multicolumn{1}{r}{$>$3600} \\
     & \textsf{SUnif} & 1     & 2.8E+55 & \multicolumn{1}{r}{---} & \multicolumn{1}{r}{---}                       & \multicolumn{1}{r}{$>$3600} \\
     &  & 0.1     & 2.8E+104 & \multicolumn{1}{r}{---} & \multicolumn{1}{r}{---}                       & \multicolumn{1}{r}{$>$3600} \\ \hline
100  & \textsf{MILP}     & \multicolumn{1}{r}{---}        & \multicolumn{1}{r}{---}                & -9.78E-01          & \multicolumn{1}{r}{---}                                      & 1.0E+00            \\
     & \textsf{Grid}     & 1          & 5.1E+03            & -9.78E-01          & 1.39E-04                                 & 2.3E-02            \\
     &                                    & 0.1        & 2.5E+22 & \multicolumn{1}{r}{---} & \multicolumn{1}{r}{---}                       & \multicolumn{1}{r}{$>$3600} \\
     & \textsf{SUnif} & 1     & 1.0E+126 & \multicolumn{1}{r}{---} & \multicolumn{1}{r}{---}                       & \multicolumn{1}{r}{$>$3600} \\
     &  & 0.1     & 1.0E+225 & \multicolumn{1}{r}{---} & \multicolumn{1}{r}{---}                       & \multicolumn{1}{r}{$>$3600} \\ \hline
500  & \textsf{MILP}     & \multicolumn{1}{r}{---}        & \multicolumn{1}{r}{---}                & -9.95E-01          & \multicolumn{1}{r}{---}                                      & 1.4E+02            \\
     & \textsf{Grid}     & 1          & 1.3E+05            & -9.95E-01          & 6.89E-05                                 & 8.2E-01            \\
     &                                    & 0.1        &  5.7E+35 & \multicolumn{1}{r}{---} & \multicolumn{1}{r}{---}                       & \multicolumn{1}{r}{$>$3600} \\
     & \textsf{SUnif} & 1     & \multicolumn{1}{r}{$>$\texttt{realmax}} & \multicolumn{1}{r}{---} & \multicolumn{1}{r}{---}                       & \multicolumn{1}{r}{$>$3600} \\
     & & 0.1     & \multicolumn{1}{r}{$>$\texttt{realmax}} & \multicolumn{1}{r}{---} & \multicolumn{1}{r}{---}                       & \multicolumn{1}{r}{$>$3600} \\ \hline
1000 & \textsf{MILP}     & \multicolumn{1}{r}{---}        & \multicolumn{1}{r}{---}                & \multicolumn{1}{r}{$>$3600} & \multicolumn{1}{r}{---}                                      & \multicolumn{1}{r}{$>$3600} \\
     & \textsf{Grid}     & 1          & \num{5.0E+05}            & -9.98E-01          & \multicolumn{1}{r}{---}                                      & 8.4E+00            \\
     &                                    & 0.1        &  5.0E+41 & \multicolumn{1}{r}{---} & \multicolumn{1}{r}{---}                                      & \multicolumn{1}{r}{$>$3600} \\
     & \textsf{SUnif} & 1     & \multicolumn{1}{r}{$>$\texttt{realmax}} & \multicolumn{1}{r}{---} & \multicolumn{1}{r}{---}                                      & \multicolumn{1}{r}{$>$3600} \\
     &  & 0.1     & \multicolumn{1}{r}{$>$\texttt{realmax}} & \multicolumn{1}{r}{---} & \multicolumn{1}{r}{---}                                      & \multicolumn{1}{r}{$>$3600} \\
     \hline
\end{tabular}
\label{tab:stQP_MILP_Grid_SimpUnif}
\end{table}

Table~\ref{tab:stQP_MILP_Grid_SimpUnif} shows the average objective values obtained by solving standard quadratic programming problems with each method and the average computational time required to solve them.
We note that although we also solved the problems with $n = 5000$, we could not solve some of the ten instances within the maximum execution time regardless of the methods used.

We were able to solve the problems with \textsf{MILP} within the maximum execution time in the case of $n \le 500$, but the computational time exceeded the maximum execution time in the case of $n = 1000$.
Whether we were able to solve the problems with \textsf{Grid} within the maximum execution time depended on the setting of $\epsilon$.
When $\epsilon = 0.1$, we were not able to solve the problems with \textsf{Grid} within the maximum execution time.
On the other hand, when $\epsilon = 1$, we were able to solve the problems with $n = 1000$ in an average time of \qty{8.4}{\second}.
In addition, the objective values obtained by using \textsf{Grid} were much more accurate than the theoretical bounds shown in \eqref{eq:Bd2002_ub_eps}.
The values in the ``Dev. from \textsf{MILP}'' column were much smaller than $\epsilon$ appearing in the respective row.
There is no reason to use \textsf{SUnif} to solve standard quadratic programming problems because the computational time of \textsf{SUnif} was always longer than that of \textsf{MILP}.

Second, to see the effectiveness of the randomized method introduced in Section~\ref{subsec:GridUnif} (denoted \textsf{GUnif}), we compare it with the other randomized method, \textsf{SUnif}, in terms of solution accuracy and computational time.
Recall that in \textsf{SUnif} we sample from the standard simplex $\Delta^{n-1}$, whereas in \textsf{GUnif} we sample points from the regular grid $\Delta_r^{n-1}$ for some positive integer $r$.
For each $n$, we solved only one of the ten instances.
We set $\epsilon$ to $1$.
For each $(n,\epsilon)$ we set $r$ as in the previous experiment and choose the number $M$ of samples to be the smallest integer greater than or equal to $c\abs{\Delta_r^{n-1}}$, where $c \in \{\num{5e-1},\num{1e-1},\num{1e-2}\}$.
We sampled $M$ points $\bm{\delta}_1,\dots,\bm{\delta}_M$ from the standard simplex $\Delta^{n-1}$ when we used \textsf{SUnif}, whereas we sampled from the grid $\Delta_r^{n-1}$ when we used \textsf{GUnif}.
Then we calculated the value $\min_{1\le i\le M}\bm{\delta}_i^\top \bm{Q}\bm{\delta}_i$.
Strictly speaking, our choice of $M$ does not follow the results in Sections~\ref{subsec:SimpUnif} and \ref{subsec:GridUnif}.
Here, we compare sampling from the standard simplex with sampling from the regular grid by matching the number of samples.

Tables~\ref{tab:stQP_GridUnif_SimpUnif_eps1} shows the objective values obtained by solving standard quadratic programming problems with each method and the computational time required to solve them.
Here, we present only the results for the case $\epsilon = 1$; similar results were also obtained for $\epsilon = 0.1$.
We did not display the results of $n = 5,10$ because in these cases the numbers of points were always less than $50$.

We observed that the objective values obtained by using \textsf{GUnif} were always better than those obtained with \textsf{SUnif}.
However, in many cases, the computational time for \textsf{GUnif} was longer than that for \textsf{SUnif}.
The MATLAB Profiler showed that our implementation of \textsf{GUnif} spent most time in the built-in function \texttt{sort}, which we used to implement uniform sampling from the regular grid.
A more efficient implementation of \textsf{GUnif} would shorten its computational time.

\begin{landscape}
\begin{table}[]
\centering
\caption{The results of solving standard quadratic programming problems by using \textsf{GUnif} and \textsf{SUnif}.
The objective values obtained by solving standard quadratic programming problems with each method are listed in the ``Obj. val.'' column and the computational time required to solve them is listed in the ``Time [\unit{\second}]'' column.
The objective value and time for \textsf{GUnif} and \textsf{SUnif} are compared for each row, respectively, and the better results are shown in bold.
The ``\textsf{GUnif}$-$\textsf{MILP}'' column denotes the difference between the objective value obtained by using \textsf{GUnif} and that obtained by using \textsf{MILP} for the same instance.
The ``$M$'' column denotes the number $M$ of points $\bm{\delta}_1,\dots,\bm{\delta}_M \in \Delta^{n-1}$ that are used to calculate an approximate optimal value $\min_{1\le i\le M} \bm{\delta}_i^\top \bm{Q}\bm{\delta}_i$ for a standard quadratic programming problem of the form \eqref{eq:stQP}.}
\begin{tabular}{r
S[table-format=1e+1]
S[table-format=6]
S[table-format=+1.2e+1]
S[table-format=+1.2e+1]
S[table-format=1.2e+1]
S[table-format=1.1e+1]
S[table-format=1.1e+1]}
\hline
     &        &           & \multicolumn{3}{c}{Obj. val.}                                                                                                                   & \multicolumn{2}{c}{Time [\unit{\second}]} \\
\cmidrule(lr){4-6}\cmidrule(lr){7-8}
\multicolumn{1}{l}{$n$}  & \multicolumn{1}{l}{$c$}    & \multicolumn{1}{l}{$M$} & \multicolumn{1}{l}{\textsf{GUnif}} & \multicolumn{1}{l}{\textsf{SUnif}} & \multicolumn{1}{l}{\textsf{GUnif}$-$\textsf{MILP}} & \multicolumn{1}{l}{\textsf{GUnif}} & \multicolumn{1}{l}{\textsf{SUnif}} \\
\hline
50   & 5E-1    & 638       & \bfseries -9.35E-01                 & -9.30E-02                          & 3.86E-02                                                              & 1.3E-02                            & \bfseries 4.7E-03                            \\
     & 1E-1    & 128       & \bfseries -6.62E-01                 & -8.64E-02                          & 3.12E-01                                                              & \bfseries 4.0E-03                   & 5.2E-03                            \\
     & 1E-2   & 13        & \bfseries -6.57E-01                 & -6.72E-02                          & 3.17E-01                                                              & \bfseries 2.3E-03                   & 2.4E-03                            \\
     \hline
100  & 5E-1    & 2525      & \bfseries -9.22E-01                 & -4.79E-02                          & 1.99E-02                                                              & 3.1E-02                            & \bfseries 9.3E-03                            \\
     & 1E-1    & 505       & \bfseries -8.99E-01                 & -3.34E-02                          & 4.31E-02                                                              & 8.8E-03                            & \bfseries 6.2E-03                            \\
     & 1E-2   & 51        & \bfseries -5.70E-01                 & -1.99E-02                          & 3.71E-01                                                              & 3.5E-03                            & \bfseries 2.5E-03                            \\
     \hline
500  & 5E-1    & 62625     & \bfseries -9.95E-01                 & -1.37E-02                          & 1.75E-03                                                              & 1.4E+00                            & \bfseries 5.0E-01                            \\
     & 1E-1    & 12525     & \bfseries -9.47E-01                 & -1.31E-02                          & 4.97E-02                                                              & 3.1E-01                            & \bfseries 9.9E-02                            \\
     & 1E-2   & 1253      & \bfseries -9.48E-01                 & -1.06E-02                          & 4.87E-02                                                              & 5.2E-02                            & \bfseries 1.6E-02                            \\
     \hline
1000 & 5E-1    & 250250    & \bfseries -9.88E-01                 & -7.60E-03                          & 4.83E-03                                                              & 1.3E+01                            & \bfseries 4.3E+00                            \\
     & 1E-1    & 50050     & \bfseries -9.82E-01                 & -5.96E-03                          & 1.10E-02                                                              & 2.7E+00                            & \bfseries 8.6E-01                            \\
     & 1E-2   & 5005      & \bfseries -9.42E-01                 & -5.25E-03                          & 5.10E-02                                                              & 2.7E-01                            & \bfseries 9.1E-02                            \\
\hline
\end{tabular}
\label{tab:stQP_GridUnif_SimpUnif_eps1}
\end{table}
\end{landscape}

\subsection{Solving general copositive programming problems}\label{subsec:NE_COPP}
In this subsection, we compare the performance of methods for solving standard quadratic programming problems when they are incorporated into Algorithm~\ref{alg:proposed} for copositive programming.
The copositive programming problem considered here is of the form \eqref{prob:COPP}.
We let $n \in \{5,10,50,100,500,1000\}$ and set $m = 5$.
For each $(n,m)$, we created ten instances of copositive programming, each generated as follows.
The set $S$ was defined as $\bbR^5$.
Each element of the vector $\bm{c} \in \bbR^5$ was independently sampled from the chi distribution with $1$ degree of freedom.
The $(1,1),\dots,(5,5)$th elements of $\bm{A}_0 \in \calS^n$ were set to $0$.
The other elements of $\bm{A}_0$ were independently sampled from the chi distribution with $1$ degree of freedom, and then $0.01$ was added to each.
For each $i \in [5]$, the $(i,i)$th element of $\bm{A}_i \in \calS^n$ was set to $1$ and the $(1,1),\dots,(i-1,i-1),(i+1,i+1),\dots,(5,5)$th elements were set to $0$.
The other elements of $\bm{A}_i$ were independently sampled from the standard normal distribution.
For each $j \in [5]$, the $(j,j)$th element of the slack matrix $\bm{A}_0 + \sum_{i=1}^5 x_i\bm{A}_i$ is $x_j$, which must be nonnegative due to copositivity.
This implies that every feasible solution $\bm{x} \in \bbR^5$ of the copositive programming problem is entrywise nonnegative.
Since $\bm{c}$ is also entrywise nonnegative, we have $\bm{c}^\top \bm{x} \ge 0$.
In addition, the nonnegativity of $\bm{A}_0$ implies that $\bm{x} = \bm{0}$ is a feasible solution, at which the objective value is $0$.
Therefore, the optimal value of the problem is $0$ and $\bm{x} = \bm{0}$ is an optimal solution.

Regardless of the subproblem solver, the initial point $\bm{x}_1 \in \bbR^5$ of Algorithm~\ref{alg:proposed} was set to the vector with all elements equal to $1$.
At each iteration indexed by $k$, we found $\bm{\delta}_k \in \Delta^{n-1}$ satisfying \eqref{eq:xi_ineq_alphaeps}.
We see that the distance between $\bm{x}_1$ and the set of optimal solutions is less than or equal to $\sqrt{5}$ since $\bm{x} = \bm{0}$ is optimal.
Based on this and \eqref{eq:N_lb}, we stopped the algorithm when the number of iterations was greater than or equal to $5L^2/\epsilon^2$ or the computational time exceeded \qty{3600}{\second}.

Following Section~\ref{subsec:NE_stQP}, we used \textsf{MILP} and \textsf{Grid} to solve a standard quadratic programming problem that appears as a subproblem at each iteration of Algorithm~\ref{alg:proposed}.
When using \textsf{MILP} to solve subproblems, we set $\alpha = 0$ and $\epsilon = 2$, and when using \textsf{Grid}, we set $\alpha = 1$ and $\epsilon = 1$, so that $\epsilon(1+\alpha) = 2$ in both cases.
When using \textsf{Grid} to solve a standard quadratic programming problem of the form \eqref{eq:stQP} at each iteration, we set $r$ to the minimum positive integer satisfying \eqref{eq:Grid_r}.

In preliminary experiments, we observed that for both \textsf{MILP} and \textsf{Grid} the computational time exceeded \qty{3600}{\second} for $n\ge 1000$.
To solve these instances, we utilized \textsf{GUnif}.
We also set $\alpha = 1$ and $\epsilon = 1$.
To solve a standard quadratic programming problem of the form \eqref{eq:stQP} at each iteration, we set $r$ to the minimum positive integer satisfying \eqref{eq:Grid_r} and sampled \num{100000} points from $\Delta_r^{n-1}$.

Existing methods based on a simplicial partition~\cite{BD2009,Zilinskas2011} and cutting-plane techniques~\cite{Bd2022_An,GBT2026} introduced in Section~\ref{sec:intro} can be used to solve copositive programming problems.
These methods seem effective for finding approximate solutions in practice, as evidenced by numerical results reported in previous studies.
However, our method differs from them in that it has a non-asymptotic convergence guarantee, which the others lack.
For this reason, we decided not to compare our method with existing methods numerically.

\begin{table}[tbp]
\centering
\caption{The results of applying Algorithm~\ref{alg:proposed} to copositive programming problems and solving the subproblems using \textsf{MILP}, \textsf{Grid}, and \textsf{GUnif}.
The average numbers of iterations are listed in the ``\#Iter.'' columns, the average objective values obtained using each method are listed in the ``Obj. val.'' column, and the average computational time required is listed in the ``Time [\unit{\second}]'' column.}
\begin{tabular}{
r
l
S[table-number-alignment=right,table-format=1.1e1]
S[table-number-alignment=right,table-format=+1.2e+1]
S[table-number-alignment=right,table-format=1.1e+1]}
\hline
\multicolumn{1}{l}{$n$}  & Method       & \multicolumn{1}{l}{\#Iter.} & \multicolumn{1}{l}{Obj. val.} & \multicolumn{1}{l}{Time [\unit{\second}]} \\
\hline
5                        & \textsf{MILP}   & 6.4E+00                          & -3.02E+00                     & 6.3E-02                          \\
                          & \textsf{Grid}   & 2.5E+01                          & -2.22E+00                     & 2.0E-02                          \\
10                       & \textsf{MILP}   & 1.1E+01                          & -3.48E+00                     & 6.6E-02                          \\
                          & \textsf{Grid}   & 4.1E+01                          & -2.62E+00                     & 1.8E-01                          \\
50                       & \textsf{MILP}   & 1.6E+01                          & -1.51E+00                     & 1.5E+00                          \\
& \textsf{Grid}   &  1               & \multicolumn{1}{r}{---}            & \multicolumn{1}{r}{$>$3600}               \\
100                      & \textsf{MILP}   & 1.9E+01                          & -1.86E+00                     & 9.3E+00                          \\
& \textsf{Grid}   &  1               & \multicolumn{1}{r}{---}            & \multicolumn{1}{r}{$>$3600}               \\
 500                      & \textsf{MILP}   & 2.8E+01                          & -1.08E+00                     & 9.2E+02                          \\
& \textsf{Grid}   &  1               & \multicolumn{1}{r}{---}            & \multicolumn{1}{r}{$>$3600}               \\
\multicolumn{1}{l}{1000} & \textsf{GUnif} & 1.2E+02                          & -1.06E+02                     & 6.7E+02                          \\
\hline
\end{tabular}
\label{tab:NE_COPP}
\end{table}

Table~\ref{tab:NE_COPP} shows the average numbers of iterations, the average objective values obtained using each method, and the average computational time required to solve copositive programming problems.
Here, we present only the results for the case $\epsilon(1+\alpha) = 2$; similar results were also obtained for $\epsilon(1+\alpha) = 0.2$.
Except for $n = 5$, the computational time for \textsf{MILP} was shorter than that for \textsf{Grid}.
We observed that when using \textsf{Grid}, solving standard quadratic programming problems was time-consuming; indeed, even the first iteration was not completed for instances that exceeded \qty{3600}{\second}.
This result contrasts with that in Table~\ref{tab:stQP_MILP_Grid_SimpUnif}, where the computational time for \textsf{Grid} was shorter than that for \textsf{MILP}.
This indicates that whether \textsf{Grid} or \textsf{MILP} is better depends on the instance.
Further investigation is needed to determine under what conditions the computational time for \textsf{Grid} is shorter than that for \textsf{MILP}.

The number of samples in \textsf{GUnif} is not consistent with Proposition~\ref{prop:GridUnif} and is much smaller than the number of points in the grid.\footnote{If we set $r$ to the minimum positive integer satisfying \eqref{eq:GUnif_r}, the number of points in $\Delta_r^{n-1}$ sometimes exceeded \texttt{realmax} in MATLAB.}
The average objective values obtained using \textsf{GUnif} were rather coarse.
However, when \textsf{MILP} and \textsf{Grid} take too long to solve problems, employing \textsf{GUnif} to obtain approximate solutions may be a viable alternative.

\subsection{Testing for non-complete positivity of exceptional doubly nonnegative matrices}\label{subsec:NE_CP}
A real symmetric matrix is said to be \emph{doubly nonnegative} if it is positive semidefinite and entrywise nonnegative.
It is well known that every completely positive matrix is doubly nonnegative, and the converse holds if and only if the order is less than or equal to $4$; see \cite[page~101]{Hall1958} and \cite{MM1962}.
We can easily check whether a given matrix is doubly nonnegative, so we are interested in deciding whether a given doubly nonnegative matrix is completely positive.
In Theorem~\ref{thm:suff_cond_not_CP}, we provided a sufficient condition for a given matrix not to be completely positive.
In this subsection, we demonstrate that this theorem is effective for doubly nonnegative matrices that are not completely positive.
For convenience, following the terminology of \cite{SZ2025}, we call a doubly nonnegative matrix that is not completely positive \emph{exceptional}.\footnote{Burer, Anstreicher, and D\"{u}r~\cite{BAD2009} refer to an exceptional doubly nonnegative matrix as a \emph{bad} matrix.}

From the numerical results in Sections~\ref{subsec:NE_stQP} and \ref{subsec:NE_COPP}, we observe that, even when using \textsf{MILP}, the standard quadratic programming problem arising at each iteration of Algorithm~\ref{alg:proposed} can be solved within a reasonable time, provided that the problem size is not large.
For this reason, in the subsequent experiments we used only \textsf{MILP} to solve standard quadratic programming problems, and thus set the violation measure $\alpha$ to $0$.
We set the error tolerance $\epsilon$ and the number of iterations as described in Remark~\ref{rem:set_eps}.
For the constant $t$, specified separately for the two experiments below, Algorithm~\ref{alg:proposed} was terminated when the number of iterations reached $1/\epsilon^2 = t^2 n^2$ or when the computational time exceeded \qty{3600}{\second}.
As mentioned in Remark~\ref{rem:certify_nonCP}, for an input matrix $\bm{C} \in \calS^n$, we also stopped Algorithm~\ref{alg:proposed} as soon as we had a matrix $\bm{X}_k$ satisfying $g(\bm{X}_k;\bm{\delta}_k) \le \epsilon$ and $\langle \bm{C},\bm{X}_k\rangle < -n\epsilon$.

\begin{table}[tbp]
\centering
\caption{The results of detecting non-complete positivity for the exceptional doubly nonnegative matrices presented in \cite[Appendix~B]{Bd2022_Simulated}. In the ``Non-CP'' column, ``Y'' denotes successful detection, whereas ``N'' indicates failure to detect non-complete positivity.}
\begin{tabular}{
l
l
S[table-number-alignment=right,table-format=1.1e1]
}
\hline
Name                                         & Non-CP & \multicolumn{1}{l}{Time [\unit{\second}]} \\
\hline
\texttt{extremal\_rand\_1}  & Y      & 2.3E+02      \\
\texttt{extremal\_rand\_2}  & Y      & 4.7E+01     \\
\texttt{extremal\_rand\_3}  & Y      & 1.6E+01      \\
\texttt{extremal\_rand\_4}  & Y      & 1.7E+02      \\
\texttt{extremal\_rand\_5}  & N      & 4.4E+02      \\
\texttt{extremal\_rand\_6}  & Y      & 8.4E+00      \\
\texttt{extremal\_rand\_7}  & Y      & 2.2E+01      \\
\texttt{extremal\_rand\_8}  & Y      & 4.3E+00      \\
\texttt{extremal\_rand\_9}  & Y      & 1.2E+01      \\
\texttt{extremal\_rand\_10} & Y      & 1.6E+01 \\
\hline
\end{tabular}
\label{tab:Bd2022}
\end{table}

First, we tried to detect non-complete positivity for the ten $6\times 6$ exceptional doubly nonnegative matrices presented in \cite[Appendix~B]{Bd2022_Simulated}.
We normalized each matrix and used the normalized matrix as the input matrix $\bm{C}$ in \eqref{prob:test_CP}.
In this experiment, we set the constant $t$ that appears in \eqref{eq:num_iter_test_CP_c} to $55$.
Table~\ref{tab:Bd2022} shows the results of detecting non-complete positivity for the ten matrices.
Except for the matrix \texttt{extremal\_rand\_5}, we successfully detected non-complete positivity for all the other matrices.

\begin{table}[tbp]
\centering
\caption{The results of detecting non-complete positivity for the exceptional doubly nonnegative matrices presented in \cite{SZ2025}. In the ``Non-CP'' column, ``Y'' denotes successful detection, whereas ``N'' indicates failure to detect non-complete positivity.}
\begin{tabular}{
l
l
l
S[table-number-alignment=right,table-format=1.1e1]
}
\hline
$n$ & $\tau$ & Non-CP & \multicolumn{1}{l}{Time [\unit{\second}]} \\
\hline
10  & 0.24       & Y      & 7.8E+01      \\
20  & 0.37       & Y      & 3.1E+02      \\
30  & 0.39       & Y      & 2.7E+03      \\
40  & 0.40       & N      & \multicolumn{1}{r}{$>$3600}      \\
\hline
\end{tabular}
\label{tab:SZ2025}
\end{table}

Next, we tried to detect non-complete positivity for exceptional doubly nonnegative matrices provided by \v{S}trekelj and Zalar~\cite{SZ2025}.
For a positive integer $m$ and $\bm{a} \in \bbR^m$, let
$\bm{C}(\bm{a}) \in \calS^n$ be the matrix whose $(i,j)$th element is
\begin{equation}
c_ic_j \int_0^1 \left(1 + 2\sum_{k=1}^m a_k\cos(2k\pi x)\right) \cos(2(i-1)\pi x)\cos(2(j-1)\pi x) dx, \label{eq:Ca_ij}
\end{equation}
where $c_i = 1$ if $i = 1$ and $c_i = \sqrt{2}$ if $i \ge 2$.
Note that we can calculate \eqref{eq:Ca_ij} analytically; see \cite[Equation~(2.3)]{SZ2025}.
Then, inspired by the method for constructing exceptional doubly nonnegative matrices presented in \cite[Equation~(1.12)]{SZ2025}, we solved the following semidefinite feasibility problem:
\begin{equation}
\begin{aligned}
&\left\langle \bm{C}_{1:5,1:5}(\bm{a}), \begin{pmatrix}
1 & -1 & 1 &  1 & -1\\
-1 & 1 & -1 & 1 & 1\\
1 & -1 & 1 & -1 & 1\\
1 & 1 & -1 & 1 & -1\\
-1 & 1 & 1 & -1 & 1
\end{pmatrix}\right\rangle = -\tau,\\
&\bm{C}(\bm{a}) \text{ is positive semidefinite},\\
&\bm{a}  \in \bbR_+^m,
\end{aligned}\label{prob:SZ2025}
\end{equation}
where $\tau$ is a given positive value and $\bm{C}_{1:5,1:5}(\bm{a})$ denotes the principal submatrix obtained by extracting the first through fifth rows and columns of $\bm{C}(\bm{a})$.
If $\bm{a}$ is a feasible solution of this problem, then $\bm{C}(\bm{a})$ is an exceptional doubly nonnegative matrix.
For each $n \in \{10,20,30,40\}$, we set $m = n + 1$ and took $\tau$ as shown in Table~\ref{tab:SZ2025} such that $\tau$ was as large as possible and Problem~\eqref{prob:SZ2025} was feasible.
After obtaining $\bm{a}$, we normalized the matrix $\bm{C}(\bm{a})$ and used the normalized matrix as the input matrix $\bm{C}$ in \eqref{prob:test_CP}.
We solved Problem~\eqref{prob:SZ2025} using the modeling language YALMIP~\cite{Lofberg2004} (version 20250626) and the MOSEK solver~\cite{MOSEK} (version 11.0.24).
In this experiment, we set the constant $t$ that appears in \eqref{eq:num_iter_test_CP_c} to $15$.
Table~\ref{tab:SZ2025} shows the results of detecting non-complete positivity for the matrices.
Despite the inability to ascertain the non-complete positivity of the matrix with $n = 40$ due to the time limit, we were able to detect it for the matrices with $n = 10,20,30$.

However, the choice of the constant $t$ in \eqref{eq:num_iter_test_CP_c} is crucial, and in the two experiments we selected $t$ by trial and error.
The smaller $t$ is, the fewer the iterations, but at the same time the harder it is to detect non-complete positivity for a matrix.
In the first experiment, setting $t=50$ failed to detect non-complete positivity for the matrix \texttt{extremal\_rand\_1}.
In the second experiment, setting $t=5$ failed to detect non-complete positivity for any of the exceptional doubly nonnegative matrices.
We need to consider how to set $t$ to balance the number of iterations and the performance of detecting non-complete positivity.

\section{Conclusion}\label{sec:conclusion}
In this paper, based on the subgradient algorithm in \cite[Equation~(3.2.24)]{Nesterov2018}, which is applicable to convex programming problems with a single nonsmooth functional constraint, we proposed an algorithm with a non-asymptotic convergence guarantee to solve copositive programming problems.
In contrast to the algorithm in \cite[Equation~(3.2.24)]{Nesterov2018}, the proposed algorithm allows us to solve the subproblem, which is a standard quadratic programming problem, inexactly at each iteration.
We discussed how to solve the standard quadratic programming problem exactly and inexactly.

Through the numerical experiments, we identified trends indicating which method is preferable under which circumstances for solving standard quadratic programming problems.
For small-scale problems, the exact method for standard quadratic programming via mixed-integer linear programming is acceptable.
For medium-scale problems, the exact method may be effective in some cases, whereas the inexact deterministic method obtained by discretizing the standard simplex using a regular grid may be effective in others.
Note, however, that the inexact method has the advantage that it does not require external solvers.
For large-scale problems where the above two methods take too long to solve, uniform sampling from the regular grid might be another option to solve standard quadratic programming problems inexactly, although its solution accuracy is coarse.

Moreover, we applied the proposed algorithm to the problem of testing complete positivity of a matrix.
By using its convergence result, we provided a sufficient condition for certifying that the matrix is not completely positive.
In the numerical experiments, we were able to detect non-complete positivity in various doubly nonnegative matrices that are not completely positive.

\vspace{0.5cm}
\noindent
{\bf Acknowledgments}
The first author is supported by JSPS Grant-in-Aid
for Research Activity Start-up JP25K23344.
The second and third authors are supported by JSPS Grant-in-Aid
for Scientific Research(B) JP23H03351.
The third author is also supported by JST CREST Grant Number JPMJCR24Q2.


\bibliographystyle{plainurl} 
\bibliography{2025_2_ref} %
\end{document}